\documentclass[a4paper]{article}
\usepackage{amsmath,amssymb,amsthm}
\usepackage{tikz}

\newcommand{\wa}{{ w^{(a)} }}
\newcommand{\iprev}{{I^{(a-1)}}}
\newcommand{\eij}{e_{ij}}
\newcommand{\eip}{e_{ip}}
\newcommand{\eiq}{e_{iq}}

\newcommand{\ejq}{e_{jq}}
\newcommand{\epi}{e_{pi}}
\newcommand{\epj}{e_{pj}}
\newcommand{\epq}{e_{pq}}

\newcommand{\eqj}{e_{qj}}

\newcommand{\eir}{e_{ir}}
\newcommand{\epr}{e_{pr}}
\newcommand{\eqr}{e_{qr}}
\newcommand{\erj}{e_{rj}}
\newcommand{\mij}{m_{ij}}

\newcommand{\miq}{m_{iq}}

\newcommand{\mpi}{m_{pi}}
\newcommand{\mpj}{m_{pj}}
\newcommand{\mpq}{m_{pq}}

\newcommand{\mpr}{m_{pr}}
\newcommand{\mrj}{m_{rj}}
\newcommand{\catc}{\mathcal{C}}
\newcommand{\one}{\mathbf{1}}
\newcommand{\lex}[1]{{ \,\underset{\mathrm{lex}}{#1}\, }}
\newcommand{\rlex}[1]{{ \,\underset{\mathrm{rlex}}{#1}\, }}
\newcommand{\const}{(\text{const.})}
\newcommand{\surj}{\twoheadrightarrow}

\newcommand{\mor}{\rightarrow}
\newcommand{\schub}{\mathfrak{S}}
\newcommand{\smod}{\mathcal{S}}
\newcommand{\der}{\partial}
\newcommand{\borel}{\mathfrak{b}}
\newcommand{\nplus}{\mathfrak{n}^+}
\newcommand{\csa}{\mathfrak{h}}
\newcommand{\ualg}{\mathcal{U}}

\newcommand{\sgn}{\mathrm{sgn}}
\newcommand{\inv}{\mathrm{code}}
\newcommand{\perm}{\mathrm{perm}}

\newcommand{\ch}{\mathrm{ch}}
\newcommand{\Ext}{\mathrm{Ext}}
\newcommand{\Hom}{\mathrm{Hom}}
\newcommand{\NN}{\ZZ_{>0}}
\newcommand{\nonneg}{\ZZ_{\geq 0}}
\newcommand{\ZZ}{\mathbb{Z}}

\newcommand{\kap}{\kappa}
\newcommand{\lmb}{\lambda}
\newcommand{\Lmb}{\Lambda}
\newcommand{\Ker}{\mathrm{Ker}}
\newcommand{\rtgeq}{{\;\underline{\triangleright}\;}}
\newcommand{\rtngeq}{{\;\not{\underline{\triangleright}}\;}}

\newtheorem{lem}{Lemma}[section]
\newtheorem{prop}[lem]{Proposition}
\newtheorem{thm}[lem]{Theorem}
\newtheorem{cor}[lem]{Corollary}
\newtheorem{Q}[lem]{Question}
\theoremstyle{definition}
\newtheorem{rem}[lem]{Remark}
\newtheorem{eg}[lem]{Example}

\title{An approach toward Schubert positivities of polynomials using Kra\'skiewicz-Pragacz modules}
\author{
Masaki Watanabe \\ 
Graduate School of Mathematical Sciences, The University of Tokyo, \\
3-8-1 Komaba Meguro-ku Tokyo 153-8914, Japan \\ \texttt{mwata@ms.u-tokyo.ac.jp}}
\date{\empty}

\begin{document}
\maketitle

\begin{center}
\textbf{Abstract}
\end{center}
\begin{center}
\begin{minipage}{0.8\textwidth}
In this paper, we investigate properties of modules introduced by Kra\'skiewicz and Pragacz which realize Schubert polynomials as their characters. 
In particular, we give some characterizations of modules having a filtration by Kra\'skiewicz-Pragacz modules. 
In finding criteria for filtrations, we calculate generating sets for the annihilator ideals of the lowest vectors in Kra\'skiewicz-Pragacz modules, and derive a projectivity result concerning Kra\'skiewicz-Pragacz modules. 
\end{minipage}

\vspace{2ex}
\end{center}
\section{Introduction}
Though Schubert polynomials originally arose from the cohomology ring of flag varieties, 
they also have purely combinatorial interests apart from the geometry of flag varieties. 
Since Schubert polynomials are a kind of generalizations of Schur functions, 
it is an interesting problem to investigate analogues of several positivity properties of Schur functions for Schubert polynomials. 
For example, it is a classical result that $\schub_u\schub_v$ is a positive sum of Schubert polynomials, 
which is usually proved using the cohomology ring of flag varieties. 
Another such problem 
is a Schubert-positivity question for the ``plethysm'' of a Schur function with a Schubert polynomial. 
For a symmetric function $s$ and a polynomial $f=x^\alpha+x^\beta+\cdots$, 
the plethysm of $s$ and $f$ is defined as $s[f]=s(x^\alpha, x^\beta, \ldots)$ (cf. \cite[\S I.8]{Mac2}). 
The question is: is $s_\sigma[\schub_w]$ a positive sum of Schubert polynomials, for all partitions $\sigma$ and permutations $w$?
In this paper, motivated by such positivity problems on Schubert polynomials, we provide some new results on the modules related with Schubert polynomials introduced by Kra\'skiewicz and Pragacz (\cite{KP0}, \cite{KP}). 

For a permutation $w$, Kra\'skiewicz and Pragacz defined a certain
representation $\smod_w$ of the Lie algebra $\borel$ of all upper triangular matrices
such that its character with respect to the subalgebra $\csa$ of all diagonal matrices
is equal to the Schubert polynomial $\schub_w$
(precise definition of $\smod_w$ will be given in the section \ref{schubmod}). 
In this paper we call these modules \textit{Kra\'skiewicz-Pragacz modules} or \textit{KP modules}. 

Since the characters of KP modules are Schubert polynomials, the problems concerning Schubert positivity are deeply related to the class of modules 
having a filtration by KP modules. 
For instance, the Schubert positivity of $\schub_u\schub_v$ and $s_\sigma[\schub_w]$ 
will follow if one shows that $\smod_u \otimes \smod_v$ and $s_\sigma(\smod_w)$ (here $s_\sigma$ denote the Schur functor), respectively, 
have such filtrations.  

KP modules are in some way similar to Demazure modules (of type A), 
the modules generated by an extremal vector in an irreducible representation of $\mathfrak{gl}_n$: 
they are both cyclic $\borel$-modules parametrized by the weight of the generators, and if the index permutation is 2143-avoiding then the KP module coincide with the Demazure modules with the same weight of the generator (note that in general they are different (see Example \ref{eg2143}): if a permutation $w$ does not avoid 2143 then there exists a \textit{strict} surjection from $\smod_w$ to the Demazure module of corresponding lowest weight). 
In this paper, we develop an analog of the theory on Demazure modules (\cite{Jos}, \cite{P}, \cite{vdK_orig}, \cite[\S 3]{vdK}) in the case of KP modules to obtain characterizations of modules having filtrations by KP modules. 

The module $\smod_w$ is generated by its lowest weight vector $u_w$. 
In this paper we first show in Section \ref{presen} that the annihilator ideal $\mathrm{Ann}_{\ualg(\nplus)}(u_w)$, 
where $\nplus$ is the Lie subalgebra of all strictly upper triangular matrices, 
is generated by the elements $\eij^{\mij(w)+1}$ ($1 \leq i < j \leq n$)
for some integers $\mij(w)$ which can be read off from $w$, 
where $e_{ij}$ denotes the $(i,j)$-th matrix unit. 
This result can be seen as a generalization of a classical result which states that
the finite dimensional irreducible representation of $\mathfrak{gl_n}$ with lowest weight $-\lmb$
can be presented as $\ualg(\nplus)/\langle e_i^{\langle \lmb,h_i \rangle+1} \rangle_{1 \leq i \leq n-1}$ as a $\ualg(\nplus)$-module. 
This result can moreover be seen as an analog of the result on Demazure modules, given by Joseph (\cite[Theorem 3.4]{Jos}), 
which states, in the $\mathfrak{gl}_n$-case, that the annihilator of the generator of the Demazure module with lowest weight $\lmb \in \ZZ^n$ 
is generated by the elements $e_{ij}^{1+\max\{0,\lmb_j-\lmb_i\}}$ $(1 \leq i < j \leq n)$. 

Using this presentation of KP modules, in section \ref{projectivity} we characterize KP modules by their projectivity in certain categories;
it is an analogue of Polo's theorem (originally for Demazure modules: see \cite{P}, \cite[\S 3]{vdK}) in the case of KP modules. 
Finally, using the results obtained so far, 
we obtain some criteria (Theorem \ref{filtr_thm}, Theorem \ref{filtr_char}) for a module to have a filtration by KP modules, 
in a way similar to the argument given by van der Kallen (\cite{vdK_orig}, \cite[\S 3]{vdK}) for Demazure modules using the method from the theory of highest-weight categories. 

The paper is organized as follows. 
In Sections \ref{prelimi} and \ref{schubmod} we recall and define some basic notations and results about Schubert polynomials and KP modules. 
In Sections \ref{presen} and \ref{prf_of_mainlem} we give a generating set for the annihilator ideal of the lowest weight vector in a KP module. 
In Section \ref{projectivity}, we introduce a new ordering on the weight lattice and show some results relating KP modules with this ordering. 
In Sections \ref{higherext} and \ref{filtr}, we obtain some characterizations of modules having a filtration by KP modules, using the results of the previous sections. 
Section \ref{questions} serves as a concluding remark by stating some future problems. 

\noindent\textbf{Acknowledgement.}
I would like to thank Katsuyuki Naoi for giving the author information on related materials. 

\section{Preliminaries}
\label{prelimi}
Let $\NN$ be the set of all positive integers and let $\nonneg$ be the set of all nonnegative integers. 
A \textit{permutation} $w$ is a bijection from $\NN$ to itself which fixes all but finitely many points. 
Let $S_\infty$ denote the group of all permutations. 
For a positive integer $n$, 
let $S_n = \{w \in S_\infty : \text{$w(i)=i$ ($i>n$)}\}$
and $S_\infty^{(n)} = \{w \in S_\infty : w(n+1)<w(n+2)<\cdots \}$. 
We sometimes write a permutation in its one-line form: i.e., write $[w(1) \, w(2) \, \cdots]$ to mean $w \in S_\infty$. 
If $w \in S_n$, we may write $[w(1) \, w(2) \, \cdots \, w(n)]$ to mean $w$. 
For $i < j$, let $t_{ij}$ denote the permutation which exchanges $i$ and $j$ and fixes all other points. 
Let $s_i=t_{i,i+1}$. 
The \textit{inversion diagram} of $w \in S_\infty$ is defined as $I(w)=\{(i,j) : i<j, w(i)>w(j)\}$. 
Let $\ell(w)=|I(w)|$ and $\sgn(w)=(-1)^{\ell(w)}$. 
For $w \in S_\infty^{(n)}$, we define $\inv(w)=(\inv(w)_1, \ldots, \inv(w)_n) \in \ZZ_{\geq 0}^n$ by $\inv(w)_i=\#\{j : i<j, w(i)>w(j)\}$: this is usually called the \textit{Lehmer code} of $w$ and it uniquely determines $w$. If $\lmb = \inv(w)$ we write $w = \perm(\lmb)$.

For a polynomial $f=f(x_1, x_2, \ldots)$ and $i \in \NN$, we define $\der_if=\frac{f-s_if}{x_i-x_{i+1}}$. 
For $w \in S_\infty$ we can assign its \textit{Schubert polynomial} $\schub_w \in \ZZ[x_1, x_2, \ldots]$, which is recursively defined by 
\begin{itemize}
\item $\schub_{w}=x_1^{m-1}x_2^{m-2} \cdots x_{m-1}$ if $w=w_0(m)=[m \; m-1 \; \cdots \; 1]$ for some $m$, and
\item $\schub_{ws_i}=\der_i\schub_w$ if $\ell(ws_i)<\ell(w)$. 
\end{itemize}
We note the fact (see eg. \cite{Mac}) that if $w \in S_n$ (resp. $S_\infty^{(n)}$)
then $\schub_w$ is a linear combination of $x_1^{a_1} \cdots x_n^{a_n}$ with $a_i \in \{0, \ldots, n-i\}$
(resp. a polynomial in $x_1, \ldots, x_n$).

Schubert polynomials satisfy the following identity known as \textit{transition}: 
\begin{prop}[{{\cite[(4.16)]{Mac}}}]
Let $w \in S_\infty \smallsetminus \{\mathrm{id}\}$. 
Let $j \in \NN$ be the maximal integer such that $w(j)>w(j+1)$ and take $k>j$ maximal with $w(j)>w(k)$. 
Let $v=wt_{jk}$. 
Let $i_1<\cdots<i_A$ be the all integers less than $j$ such that $\ell(vt_{i_aj})=\ell(v)+1$, and let $\wa=vt_{i_aj}$. 
Then
\[
\schub_w=x_j\schub_v+\sum_{a=1}^A \schub_{\wa}. 
\]
\label{transition}
\end{prop}
Note that if $w \in S_\infty^{(n)}$, $v$ and $w^{(1)}, \ldots, w^{(A)}$ in the proposition above are also in $S_\infty^{(n)}$. Note also that $\inv(v)=\inv(w)-\epsilon_j$, where $\epsilon_j=(0, \ldots, 0, 1, 0, \ldots, 0)$ with $1$ at the $j$-th position. 
.

Hereafter in this paper, we fix a positive integer $n$. 
Let $K$ be a field of characteristic zero. 
Let $\borel$ be the Lie algebra of all $n \times n$ upper triangular $K$-matrices
and let $\csa \subset \borel$ be the subalgebra of all diagonal matrices. 
Let $\ualg(\borel)$ be the universal enveloping algebra of $\borel$. 
For a $\ualg(\borel)$-module $M$ and $\lmb = (\lmb_1, \ldots, \lmb_n) \in \ZZ^n$, 
let $M_\lmb = \{m \in M : hm=\langle \lmb,h \rangle m \;\text{($\forall h \in \csa$)}\}$ where $\langle \lmb,h \rangle = \sum \lmb_i h_i$. 
$M_\lmb$ is called the \textit{weight space of weight $\lmb$} or \textit{$\lmb$-weight space}, and elements of $M_\lmb$ are said to \textit{have weight $\lmb$}. 
If $M_\lmb \neq 0$ then $\lmb$ is said to be a \textit{weight} of $M$. 
If $M$ is the direct sum of its weight spaces and each weight space has finite dimension, then $M$ is said to be a \textit{weight module}
and we define $\ch(M)=\sum_{\lmb} \dim M_\lmb x^\lmb$ where $x^\lmb=x_1^{\lmb_1} \cdots x_n^{\lmb_n}$. 
For $1 \leq i<j \leq n$, let $\eij \in \borel$ be the matrix with $1$ at the $(i,j)$-position and all other coordinates $0$. 
It is easy to see that if $M$ is a $\ualg(\borel)$-module and $x \in M_\lmb$, 
then $\eij x \in M_{\lmb+\epsilon_i-\epsilon_j}$, 

For $\lmb \in \ZZ^n$, let $K_\lmb$ denote the one-dimensional $\ualg(\borel)$-module where $h \in \csa$ acts by $\langle \lmb,h \rangle$ and $e_{ij}$ acts by $0$. 
Note that every finite-dimensional weight module admits a filtration by these one dimensional modules. 

\section{Kra\'skiewicz-Pragacz modules}
\label{schubmod}
In \cite{KP0} and \cite{KP}, Kra\'skiewicz and Pragacz defined certain $\ualg(\borel)$-modules which we call here \textit{Kra\'skiewicz-Pragacz modules} or \textit{KP modules}. 
Here we use the following definition. 
Let $w \in S_\infty^{(n)}$. 
Let $K^n=\bigoplus_{1 \leq i \leq n} K u_i$ be the vector representation of $\borel$. 
For each $j \in \NN$, let $l_j = l_j(w) = \#\{i : (i,j) \in I(w)\}$, $\{i : (i,j) \in I(w)\}=\{i_1, \ldots, i_{l_j}\}$ ($i_1<\cdots<i_{l_j}$), 
and $u_w^{(j)}=u_{i_1} \wedge \cdots \wedge u_{i_{l_j}} \in \bigwedge^{l_j} K^n$. 
Note that $u_w^{(j)} \in \bigwedge^{l_j} K^{\min\{n,j-1\}}$. 
Let $u_w=u_w^{(1)} \otimes u_w^{(2)} \otimes \cdots \in \bigwedge^{l_1} K^n \otimes \bigwedge^{l_2} K^n \otimes \cdots$. 
Then the KP module $\smod_w$ associated to $w$ is defined as $\smod_w=\ualg(\borel)u_w$. 

\begin{rem}
It is also possible to define KP modules using so-called \textit{Rothe diagram} $D(w)=\{(i, w(j)) : i<j, w(i)>w(j)\}$ of $w$ instead of $I(w)$. Since $I(w)$ and $D(w)$ differ only by a rearrangement of columns it does not matter which to use. $D(w)$ has an advantage that it is easier to see with hand what the diagram looks like: drawing rays downward and to the right from the positions $(i, w(i))$ ($i=1, 2, \ldots$) and then the remaining boxes give $D(w)$ (see the figure below). Also, in \cite{FGRS} a basis for $\smod_w$ is constructed using certain labellings of Rothe diagram. 
\end{rem}

\begin{center}
{\unitlength 0.1in%
\begin{picture}( 28.0000, 10.9000)(  2.0000,-12.0000)%
%
\special{pn 8}%
\special{pa 310 110}%
\special{pa 490 110}%
\special{pa 490 290}%
\special{pa 310 290}%
\special{pa 310 110}%
\special{pa 490 110}%
\special{fp}%
%
\special{pn 8}%
\special{pa 310 310}%
\special{pa 490 310}%
\special{pa 490 490}%
\special{pa 310 490}%
\special{pa 310 310}%
\special{pa 490 310}%
\special{fp}%
%
\special{pn 8}%
\special{pa 510 310}%
\special{pa 690 310}%
\special{pa 690 490}%
\special{pa 510 490}%
\special{pa 510 310}%
\special{pa 690 310}%
\special{fp}%
%
\special{pn 8}%
\special{pa 710 310}%
\special{pa 890 310}%
\special{pa 890 490}%
\special{pa 710 490}%
\special{pa 710 310}%
\special{pa 890 310}%
\special{fp}%
%
\special{pn 8}%
\special{pa 710 710}%
\special{pa 890 710}%
\special{pa 890 890}%
\special{pa 710 890}%
\special{pa 710 710}%
\special{pa 890 710}%
\special{fp}%
%
\special{pn 8}%
\special{pa 1910 110}%
\special{pa 2090 110}%
\special{pa 2090 290}%
\special{pa 1910 290}%
\special{pa 1910 110}%
\special{pa 2090 110}%
\special{fp}%
%
\special{pn 8}%
\special{pa 1910 310}%
\special{pa 2090 310}%
\special{pa 2090 490}%
\special{pa 1910 490}%
\special{pa 1910 310}%
\special{pa 2090 310}%
\special{fp}%
%
\special{pn 8}%
\special{pa 2310 310}%
\special{pa 2490 310}%
\special{pa 2490 490}%
\special{pa 2310 490}%
\special{pa 2310 310}%
\special{pa 2490 310}%
\special{fp}%
%
\special{pn 8}%
\special{pa 2510 310}%
\special{pa 2690 310}%
\special{pa 2690 490}%
\special{pa 2510 490}%
\special{pa 2510 310}%
\special{pa 2690 310}%
\special{fp}%
%
\special{pn 8}%
\special{pa 2310 710}%
\special{pa 2490 710}%
\special{pa 2490 890}%
\special{pa 2310 890}%
\special{pa 2310 710}%
\special{pa 2490 710}%
\special{fp}%
%
\special{pn 8}%
\special{pa 2200 200}%
\special{pa 3000 200}%
\special{dt 0.045}%
\special{pa 2200 200}%
\special{pa 2200 1200}%
\special{dt 0.045}%
\special{pa 2800 400}%
\special{pa 3000 400}%
\special{dt 0.045}%
\special{pa 2800 400}%
\special{pa 2800 1200}%
\special{dt 0.045}%
\special{pa 2000 600}%
\special{pa 3000 600}%
\special{dt 0.045}%
\special{pa 2000 600}%
\special{pa 2000 1200}%
\special{dt 0.045}%
\special{pa 2600 800}%
\special{pa 3000 800}%
\special{dt 0.045}%
\special{pa 2600 800}%
\special{pa 2600 1200}%
\special{dt 0.045}%
\special{pa 2400 1000}%
\special{pa 3000 1000}%
\special{dt 0.045}%
\special{pa 2400 1000}%
\special{pa 2400 1200}%
\special{dt 0.045}%
%
\special{pn 4}%
\special{sh 1}%
\special{ar 200 200 8 8 0  6.28318530717959E+0000}%
\special{sh 1}%
\special{ar 600 200 8 8 0  6.28318530717959E+0000}%
\special{sh 1}%
\special{ar 800 200 8 8 0  6.28318530717959E+0000}%
\special{sh 1}%
\special{ar 800 600 8 8 0  6.28318530717959E+0000}%
\special{sh 1}%
\special{ar 600 600 8 8 0  6.28318530717959E+0000}%
\end{picture}}%

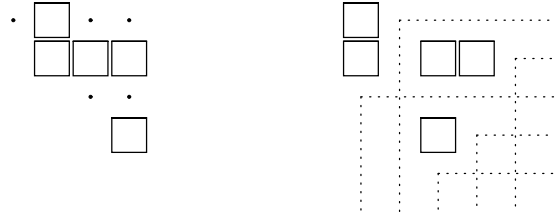
\noindent{\scriptsize Figure 1: inversion diagram and Rothe diagram of the same permutation $[25143]$. }
\end{center}

KP modules have the following property: 
\begin{thm}[{{\cite[Remark 1.6 and Theorem 4.1]{KP}}}]
$\smod_w$ is a weight module and $\ch(\smod_w)=\schub_w$. 
\end{thm}
\begin{eg}
If $w=s_i$, then $I(s_i)=\{(i,i+1)\}$, $u_{s_i}=u_i$ and $\smod_{s_i}=\bigoplus_{1 \leq j \leq i} Ku_j = K^i$. So $\ch(\smod_{s_i}) = x_1+\cdots+x_i = \schub_{s_i}$. 
\end{eg}
\begin{eg}
More generally, if $w$ is grassmannian, i.e. there exists a $k$ such that $w(1)<\cdots<w(k)$ and $w(k+1) < w(k+2) < \cdots$, then the inversion diagram $I(w)$ of $w$ is a ``French-notation'' Young diagram (see Figure 2). Thus in this case, $u_w$ is a lowest-weight vector in certain irreducible representation of $\mathfrak{gl}_k$, and $\smod_w$ is equal to this representation (seen as a representation of $\borel_n$ through the morphism $\borel_n \ni \epq \mapsto \begin{cases} \epq & (q \leq k) \\ 0 & (q > k)\end{cases} \in \mathfrak{gl}_k$). 
This reflects the fact that the Schubert polynomial indexed by a grassmannian permutation is a Schur polynomial. 
\end{eg}
\begin{center}
{\unitlength 0.1in%
\begin{picture}(  8.9000,  8.0000)(  4.0000,-10.0000)%
%
\special{pn 4}%
\special{sh 1}%
\special{ar 400 200 8 8 0  6.28318530717959E+0000}%
\special{sh 1}%
\special{ar 600 200 8 8 0  6.28318530717959E+0000}%
\special{sh 1}%
\special{ar 800 200 8 8 0  6.28318530717959E+0000}%
\special{sh 1}%
\special{ar 1000 200 8 8 0  6.28318530717959E+0000}%
\special{sh 1}%
\special{ar 1200 200 8 8 0  6.28318530717959E+0000}%
\special{sh 1}%
\special{ar 600 400 8 8 0  6.28318530717959E+0000}%
\special{sh 1}%
\special{ar 1000 400 8 8 0  6.28318530717959E+0000}%
\special{sh 1}%
\special{ar 1200 400 8 8 0  6.28318530717959E+0000}%
\special{sh 1}%
\special{ar 1000 800 8 8 0  6.28318530717959E+0000}%
\special{sh 1}%
\special{ar 1200 800 8 8 0  6.28318530717959E+0000}%
\special{sh 1}%
\special{ar 1200 1000 8 8 0  6.28318530717959E+0000}%
%
\special{pn 8}%
\special{pa 710 310}%
\special{pa 890 310}%
\special{pa 890 490}%
\special{pa 710 490}%
\special{pa 710 310}%
\special{pa 890 310}%
\special{fp}%
%
\special{pn 8}%
\special{pa 710 510}%
\special{pa 890 510}%
\special{pa 890 690}%
\special{pa 710 690}%
\special{pa 710 510}%
\special{pa 890 510}%
\special{fp}%
%
\special{pn 8}%
\special{pa 910 510}%
\special{pa 1090 510}%
\special{pa 1090 690}%
\special{pa 910 690}%
\special{pa 910 510}%
\special{pa 1090 510}%
\special{fp}%
%
\special{pn 8}%
\special{pa 1110 510}%
\special{pa 1290 510}%
\special{pa 1290 690}%
\special{pa 1110 690}%
\special{pa 1110 510}%
\special{pa 1290 510}%
\special{fp}%
\end{picture}}%

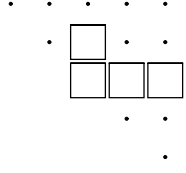
\noindent{\scriptsize Figure 2: inversion diagram of a grassmannian permutation $[136245]$ is a French-style Young diagram of shape $(3,1)$. }
\end{center}
\begin{eg}
\label{eg2143}
More generally, if $w$ is $2143$-avoiding, then it can be seen that $u_w$ is an extremal vector in an irreducible representation of $\mathfrak{gl}_n$ (using the fact (\cite[(1.27)]{Mac}) that the rows of $I(w)$ for $2143$-avoiding $w$ is totally preordered by inclusion). Thus in this case the corresponding KP module $\smod_w$ is isomorphic to a Demazure module of $\borel$: i.e. a module generated by an extremal vector of an irreducible representation of $\mathfrak{gl}_n$. 
Note that this corresponds to the result of Lascoux and Schutzenberger (\cite[Theorem 5]{LS}, \cite[Corollary 10.5.2]{Las}) that Schubert polynomials with 2143-avoiding indices are equal to certain key polynomials. 

On the other hand, consider $w=[2143]$. Then $I(w)=\{(1,2),(3,4)\}$, $u_w=u_1 \otimes u_3$,  $\smod_w=\bigoplus_{1 \leq i \leq 3} K(u_1 \otimes u_i) = K^1 \otimes K^3$
and $\ch(\smod_w)=x_1(x_1+x_2+x_3)=\schub_w$. 
Note that in this case $\smod_w$ is \textit{not} isomorphic to the Demazure module with the same lowest weight: $\smod_w$ is three-dimensional while the Demazure module with the same lowest weight is two-dimensional. 
\footnote{
The KP module $\smod_{[2143]}$ in this example is, if not seen as a $\ualg(\borel)$-module but as a $\ualg(\nplus)$-module, isomorphic to a Demazure module (say $V(0,0,1)$); thus the results such as Theorem \ref{mainthm} for such kind of KP modules follow from known results on Demazure modules. But in fact there also exist KP modules which are, even as $\ualg(\nplus)$-modules, not isomorphic to any Demazure modules.  An example is $\smod_{[13254]} \cong K^2 \otimes K^4$. 
}
In general, $\smod_w$ is isomorphic to the Demazure module $V(\inv(w))$ with lowest weight $\inv(w)$ if and only if $w$ is 2143-avoiding. We also note here that there always exists a surjection from $\smod_w$ to $V(\inv(w))$: this can be seen using the result from the next section and \cite[Theorem 3.4]{Jos}. 
\end{eg}

In this paper we have to slightly extend the notion of Schubert polynomials and KP modules. 
For $\lmb=(\lmb_1, \ldots, \lmb_n) \in \ZZ^n$, we define the Schubert polynomial and the KP module associated to $\lmb$ as follows. 
For $\lmb \in \nonneg^n$, let $\schub_\lmb=\schub_w$ and $\smod_\lmb=\smod_w$ where $w=\perm(\lmb)$. 
For a general $\lmb \in \ZZ^n$, take $k \in \ZZ$ so that $\lmb+k\one \in \nonneg^n$, where $\one=(1, \ldots, 1)$, 
and we define $\schub_\lmb=x^{-k\one}\schub_{\lmb+k\one}$ and $\smod_\lmb=K_{-k\one} \otimes \smod_{\lmb+k\one}$. 
Note that this definition does not depend on the choice of $k$, 
since if $\perm(\lmb)=w$, then $\perm(\lmb+\one)=\tilde{w}=[w(1)+1 \; \cdots \; w(n)+1 \,\; 1 \,\; w(n+1)+1 \; \cdots]$, and $\schub_{\tilde{w}}=x^\one\schub_w$ and $\smod_{\tilde{w}}=K_{\one} \otimes \smod_w$ hold for them. 
It then follows from the theorem above that $\smod_\lmb$ is a weight module and $\ch(\smod_\lmb)=\schub_\lmb$ for all $\lmb \in \ZZ^n$. 
Note that, since $\smod_\lmb$ is generated by an element of weight $\lmb$, if $(\smod_\lmb)_\mu \neq 0$ (i.e. if $x^\mu$ appears in $\schub_\lmb$ with nonzero coefficient) then $\mu \rtgeq \lmb$, 
where $\rtgeq$ denote the dominance order: $\mu \rtgeq \lmb$ iff $\mu-\lmb=\sum_{i=1}^{n-1} a_i(\epsilon_i-\epsilon_{i+1})$ for some $a_1, \ldots, a_{n-1} \in \nonneg$. 
We also note here that for any $\mu, \nu \in \ZZ^n$, the number of $\lmb \in \ZZ^n$ with $\mu \rtgeq \lmb \rtgeq \nu$ is finite. 

A \textit{KP filtration} of a weight $\borel$-module $M$
is a sequence $0=M_0 \subset \cdots \subset M_r=M$ of weight $\borel$-modules
such that each $M_i / M_{i-1}$ is isomorphic to some KP module $\smod_{\lmb^{(i)}}$. 
Note that if $M$ has a KP filtration then $\ch(M)$ is a positive sum of Schubert polynomials. 


\section{Annihilator of the lowest weight vector}
\label{presen}
For $w \in S_\infty^{(n)}$ and $1 \leq i<j \leq n$, 
let $C_{ij}(w) = \{ k : \mbox{$(i,k) \not\in I(w)$, $(j,k) \in I(w)$}\} = \{k : k > j, w(i) < w(k) < w(j)\}$ and let $\mij(w)=|C_{ij}(w)| = \#\{k>j : w(i)<w(k)<w(j)\}$ (in particular, $\mij(w)=0$ if $w(i)>w(j)$). 
Since $\eij^2 u_w^{(k)}=0$ for $k \in C_{ij}(w)$ and $\eij u_w^{(k)}=0$ for $k \not\in C_{ij}(w)$, 
we see that $\eij^{\mij(w)+1}$ annihilates $u_w = u_w^{(1)} \otimes u_w^{(2)} \otimes \cdots$. 
Let $I_w$ denote the left ideal of $\ualg(\borel)$ generated by $h-\langle \inv(w), h\rangle$ ($h \in \csa$) and $\eij^{\mij(w)+1}$ ($i<j$).
Then, by the observation above and the fact that $u_w$ has weight $\inv(w)$, there is a unique surjective morphism of $\ualg(\borel)$-modules from $\ualg(\borel) / I_w$ to $\smod_w$
sending $1 \bmod I_w$ to $u_w$. 
We show the following: 
\begin{thm}
The surjection above is an isomorphism. 
\label{mainthm}
\end{thm}
\begin{rem}
It is also possible to define $u_D$ and $\smod_D$ for a general finite subset $D \subset \{1, \ldots, n\} \times \NN$
as in the same way we defined KP modules
($\smod_D$ is often called the \textit{flagged Schur module} associated to $D$, see eg. \cite[\S 7]{Mag}; 
the equivalence of the definition there and our definition can be checked by the same argument as in \cite[Remark 1.6]{KP}). 
Again in this setting, if we let $\mij(D)=\#\{p : (i,p) \not \in D, (j,p) \in D\}$ 
and $\lmb_i=\#\{p : (i,p) \in D\}$, 
then $\eij^{\mij(D)+1}$ ($i<j$) and $h-\langle \lmb,h\rangle$ ($h \in \csa$) annihilate $u_D$, 
and therefore we have a surjective morphism $\ualg(\borel)/I_D \surj \smod_D$ 
where $I_D$ is the left ideal generated by these elements. 
But this is not an isomorphism for general $D$: for example, if $D=\{(2,1),(3,2)\}$, 
then $\ch(\ualg(\borel)/I_D)=x_2x_3+x_1x_3+x_2^2+2x_1x_2+x_1^2+x_1x_2^2x_3^{-1}$
while $\ch(\smod_D)=x_2x_3+x_1x_3+x_2^2+2x_1x_2+x_1^2$.
\end{rem}

The theorem can be reduced to the following lemma, which will be proved in the next section: 
\begin{lem}
Let $w \in S_\infty^{(n)} \smallsetminus \{\mathrm{id}\}$ and take $j, i_1, \ldots, i_A$ and $v, w^{(1)}, \ldots, w^{(A)}$ as in Proposition \ref{transition}. 
Let $x_a=e_{i_aj}^{m_{i_aj}(v)+1}$ for $a=1, \ldots, A$. Let $I^{(0)}=I_w$ and $I^{(a)}=I^{(a-1)}+\ualg(\borel)x_a$ for $a=1, \ldots, A$.  
Also let $I'_v$ be the left ideal of $\ualg(\borel)$ generated by $h-\langle \inv(w), h\rangle = h-\langle \inv(v)+\epsilon_j, h\rangle$ $(h \in \csa)$ and $\eij^{\mij(v)+1}$ $(i<j)$, so $\ualg(\borel) / I'_v \cong \ualg(\borel) / I_v \otimes K_{\epsilon_j}$. 
Then $I'_v \subset I^{(A)}$ and $I_{w^{(a)}}x_a \subset I^{(a-1)}$ for $a=1, \ldots, A$. 
\label{mainlemma2}
\end{lem}

Here we show Theorem \ref{mainthm} assuming Lemma \ref{mainlemma2}. 
Let $d_w = \dim \ualg(\borel)/I_w$. The conclusion of Lemma \ref{mainlemma2} claims that there exist surjective morphisms $\ualg(\borel)/I_v \otimes K_{\epsilon_j} \cong \ualg(\borel)/I'_v \surj \ualg(\borel)/I^{(A)} : (x \bmod I'_v) \mapsto (x \bmod I^{(A)})$ and $\ualg(\borel)/I_{w^{(a)}} \surj I^{(a)}/I^{(a-1)} : (x \bmod I_{w^{(a)}}) \mapsto (xx_a \bmod I^{(a-1)})$ (note that $xx_a \in I^{(a)}$ since $x_a \in I^{(a)}$). 
Thus $\ualg(\borel)/I_w = \ualg(\borel)/I^{(0)}$ has a quotient filtration $\ualg(\borel)/I^{(0)} \surj \ualg(\borel)/I^{(1)} \surj \cdots \surj \ualg(\borel)/I^{(A)} \surj 0$ with each subquotient being a quotient of $\ualg(\borel)/I_{w^{(1)}}, \cdots, \ualg(\borel)/I_{w^{(A)}}$ and $\ualg(\borel)/I_v \otimes K_{\epsilon_j}$ respectively. 
Therefore $d_w \leq d_{w^{(1)}}+\cdots+d_{w^{(A)}}+d_v$. 
So, by Proposition \ref{transition} and induction on lexicographic ordering on $(\ell(w), \schub_w(1))$, we see that $d_w \leq \schub_w(1)$ hold for any $w$. But on the other hand, we have a surjection $\ualg(\borel) / I_w \surj \smod_w$ and thus $d_w \geq \dim \smod_w = \schub_w(1)$. Thus $d_w = \schub_w(1)$ and the surjection above must be an isomorphism. This completes the proof of Theorem \ref{mainthm}. 

\section{{Proof of Lemma \ref{mainlemma2}}}
\label{prf_of_mainlem}
Throughout this section, let $w \in S_\infty^{(n)} \smallsetminus \{\mathrm{id}\}$ and take $j, i_1, \ldots, i_A$, $v, w^{(1)}, \ldots, w^{(A)}$ as in Proposition \ref{transition}. 
Take $x_1, \ldots, x_a$ and $I^{(0)}, \ldots, I^{(A)}$ as in Lemma \ref{mainlemma2}. 
Let $m_{pq}=m_{pq}(v)$ for $1 \leq p<q \leq n$. 
For $x, y, \ldots, z \in \ualg(\borel)$, let $\langle x, y, \ldots, z \rangle$ denote the left ideal of $\ualg(\borel)$ generated by $x, y, \ldots, z$. 

To make the calculations simple, we use the following basic fact from the representation theory of semisimple Lie algebras: 
\begin{prop}
Let $\nplus_3 = Ke_{12} \oplus Ke_{13} \oplus Ke_{23}$ be the Lie algebra of all $3 \times 3$ strictly upper triangular matrices 
which acts on $K^3=Ku_1 \oplus Ku_2 \oplus Ku_3$ and $\bigwedge^2 K^3 = K(u_1 \wedge u_2) \oplus K(u_1 \wedge u_3) \oplus K(u_2 \wedge u_3)$ in the usual way. 
Then for $a, b \geq 0$, the $\ualg(\nplus_3)$-module generated by $(u_2 \wedge u_3)^a \otimes u_3^b \in S^a(\bigwedge^2 K^3) \otimes S^b(K^3)$ 
$(\text{$S^\bullet$ denotes the symmetric product})$
is isomorphic to $\ualg(\nplus_3)/I_{a,b}$ where $I_{a,b}$ is the left ideal of $\ualg(\nplus_3)$ generated by $e_{12}^{a+1}$ and $e_{23}^{b+1}$. 
\label{u3prop}
\end{prop}
\begin{proof}
First note that $(u_2 \wedge u_3)^a \otimes u_3^b$ is a lowest weight vector of an irreducible representation of $\mathfrak{sl}_3$: i.e.  $\ualg(\nplus_3) ((u_2 \wedge u_3)^a \otimes u_3^b)$ is an irreducible representation of $\mathfrak{sl}_3$. Thus the claim is merely a well-known fact that a finite-dimensional irreducible representation $V(\lmb)$, with lowest weight $\lmb$, of a finite-dimensional semisimple Lie algebra $\mathfrak{g}$ with simple root system $\Delta$ and upper-triangular part $\nplus$ is isomorphic to $\ualg(\mathfrak{\nplus})/\langle e_\alpha^{\langle \lmb, h_\alpha \rangle}\rangle_{\alpha \in \Delta}$ as $\ualg(\nplus)$-modules (\cite[Theorem 21.4]{Hum}). 
\end{proof}
From this proposition, we have the following: 
\begin{lem}
Let $f(x,y,z)$ be a polynomial (in non-commutative variables) and let $a, b \geq 0$. 
If $f(e_{12}, e_{13}, e_{23})((u_2 \wedge u_3)^a \otimes u_3^b) = 0$, 
then for $1 \leq p < q < r \leq n$, $f(\epq, \epr, \eqr) \in \langle \epq^{a+1}, \eqr^{b+1} \rangle$. 
\label{hoge}
\end{lem}
\begin{proof}
From Proposition \ref{u3prop} we have $f(e_{12}, e_{13}, e_{23}) \in \ualg(\nplus_3)e_{12}^{a+1} + \ualg(\nplus_3)e_{23}^{b+1}$, 
i.e. $f(e_{12},e_{13},e_{23})=g(e_{12},e_{13},e_{23})e_{12}^{a+1}+h(e_{12},e_{13},e_{23})e_{23}^{b+1}$ for some $g$ and $h$. 
Then $f(e_{pq},e_{pr},e_{qr})=g(e_{pq},e_{pr},e_{qr})e_{pq}^{a+1}+h(e_{pq},e_{pr},e_{qr})e_{qr}^{b+1} \in \langle e_{pq}^{a+1}, e_{qr}^{b+1} \rangle$. 

\end{proof}
With this lemma in hand, it is easy to prove the following: 
\begin{lem}
For $1 \leq p < q < r \leq n$ and $N, M, N', M' \geq 0$, 
\begin{enumerate}
\item $\epr^N\eqr^M \equiv 0 \pmod{\langle \epq^{N'+1}, \eqr^{M'+1} \rangle}$ if $N+M > N'+M'$. 
\item $\epq^N\epr^M \equiv 0 \pmod{\langle \epq^{N'+1}, \eqr^{M'+1} \rangle}$ if $N+M > N'+M'$. 
\item $\epr^N \equiv \frac{(-1)^N}{N!}\eqr^N\epq^N \pmod{\langle \epq^{M+1}, \eqr \rangle}$ (and in fact $\mathrm{mod} \; \langle e_{qr} \rangle$, although we do not need it here). 
\item $\epr^N \equiv \frac{1}{N!}\epq^N\eqr^N \pmod{\langle \epq, \eqr^{M+1} \rangle}$ (and $\mathrm{mod} \; \langle e_{pq} \rangle$: we do not need it here). 
\item $\epq^{N+M+1}\eqr^M \equiv 0 \pmod{\langle \epq^{N+1}, \eqr^{M+1} \rangle}$. 
\item $\epq^N\eqr^M \equiv 0 \pmod{\langle \epq, \epr^N, \eqr^{M+1} \rangle}$. 
\end{enumerate}
\label{u3lem}
\end{lem}
\begin{proof}
(1)-(5) follows from straightforward calculations checking the condition of Lemma \ref{hoge}. 
(6) also follows from Lemma \ref{hoge}, 
since $e_{12}^Ne_{23}^Mu_3^M = \const \cdot u_1^Nu_2^{M-N} = \const \cdot e_{23}^{M-N}e_{13}^Nu_3^M$
so $\epq^N \eqr^M - \const \cdot \eqr^{M-N} \epr^N \in \langle \epq, \eqr^{M+1} \rangle$. 
\end{proof}

Let us move on to the proof of Lemma \ref{mainlemma2}. 
First we prove $I'_v \subset I^{(A)}$. 
Since $h-\langle \inv(w),h\rangle \in I_w \subset I^{(A)}$, 
it suffices to show $\epq^{\mpq+1} \in I^{(A)}$ for all $1 \leq p < q \leq n$. 
If $q \neq j$, we have $\mpq=\mpq(w)$ so $\epq^{\mpq+1} \in I_w \subset I^{(A)}$. 
If $q=j$ and $v(p)>v(j)$, then $\mpq=0=\mpq(w)$ (note that, by the choice of $k$, there does not exist $r>j$ such that $w(k)<w(r)<w(j)$), and thus again $\epq^{\mpq+1} \in I_w \subset I^{(A)}$. 
If $q=j$ and $p=i_a$, we have $e_{i_aj}^{m_{i_aj}+1} = x_a \in I^{(a)} \subset I^{(A)}$. 
Otherwise (i.e. if $q=j$, $v(p)<v(j)$ and $p \neq i_1, \ldots, i_A$), the conclusion follows from the following lemma: 
\begin{lem}
Let $p < j$, $v(p) < v(j)$ and $p \neq i_1, \ldots, i_A$. Then
\begin{enumerate}
\item There exists some $a \in \{1, \ldots, A\}$ such that $v(i_a)>v(p)$. 
\item Let $a \in \{1, \ldots, A\}$ be the maximal index such that $v(i_a)>v(p)$. 
Then $\epj^{\mpj+1} \in I^{(a)}$. 
\end{enumerate}
\label{jiminidaiji}
\end{lem}
\begin{proof}
(1): 
By the assumptions we have $\ell(vt_{pj}) > \ell(v)+1$, and thus there exists an $i$ such that $p < i < j$ and $v(p)<v(i)<v(j)$. Take $i$ to be maximal among such. Then there does not exist $i'$ such that $i<i'<j$ and $v(i)<v(i')<v(j)$, and thus $\ell(vt_{ij})=\ell(v)+1$. Therefore $i$ is in $\{i_1, \ldots, i_A\}$. This shows (1) since $v(i)>v(p)$. 

(2): 
Let $i=i_a$. 
First we claim that there exists no $r$ such that $i<r<j$ and $v(p)<v(r)<v(i)$. 
Suppose such $r$ exists. Take $r$ to be maximal among such. Then by the same argument as in (1) we see that $r$ is in $\{i_1, \ldots, i_A\}$, and since $i<r$ we have $r=i_b$ for some $b>a$. This contradicts to the choice of $a$. 

From the claim we see $\mpi=\#\{r>i : v(p)<v(r)<v(i)\} = \#\{r>j : v(p)<v(r)<v(i)\} = \mpj-\mij$. 
So by Lemma \ref{u3lem}(1), $\epj^{\mpj+1} \in \langle \epi^{\mpi+1}, \eij^{\mij+1} \rangle$. 
Since $\epi^{\mpi+1} \in I_w \subset I^{(a)}$ and $\eij^{\mij+1} = x_a \in I^{(a)}$ we are done. 
\end{proof}

Let us now prove $I_{w^{(a)}}x_a \subset \iprev$ ($a=1, \ldots, A$). 
Fix $a \in \{1, \ldots, A\}$ and let $i=i_a$. 
We want to prove $(h-\langle \inv(\wa),h \rangle)x_a \in \iprev$ for all $h \in \csa$ 
and $\epq^{\mpq(\wa)+1}x_a \in \iprev$ for all $p<q$. 
We first check $(h-\langle \inv(\wa),h \rangle)x_a \in \iprev$, i.e., the element $x_a \bmod \iprev$ has weight $\inv(\wa)$. 
It is easy to see that  
$\inv(\wa)=\inv(v)+(\mij+1)\epsilon_i-\mij\epsilon_j=\inv(w)+(\mij+1)(\epsilon_i-\epsilon_j)$. On the other hand, $x_a \bmod \iprev = \eij^{\mij+1} \bmod \iprev$ has weight $\inv(w)+(\mij+1)(\epsilon_i-\epsilon_j)$ since $1 \bmod \iprev$ has weight $\inv(w)$ and $\eij$ shifts the weight by $\epsilon_i-\epsilon_j$. This shows the claim. 

We now check $\epq^{\mpq(\wa)+1}x_a=\epq^{\mpq(\wa)+1}\eij^{\mij+1}$ is in $\iprev$ for all $1 \leq p<q \leq n$, case by case. 
First note that, by Lemma \ref{jiminidaiji} and the consideration before that lemma, 
$\epq^{\mpq+1} \in \iprev$ unless $q=j$ and $v(p) \leq v(i)$, 
and in such case we see $\epq^{\mpq+2} = \epq^{\mpq(w)+1} \in I_w \subset \iprev$. 
Also note that there does not exist an $r$ such that $i<r<j$ and $v(i)<v(r)<v(j)$, since $\ell(vt_{ij})=\ell(v)+1$. 

\begin{itemize}
\item $q>j$ : 
In this case we have $\mpq(\wa)=0=\mpq(w)$, 
since both $w$ and $\wa$ are increasing from $(j+1)$-th position
and thus there are no $r>q$ with $w(r)<w(q)$ or $\wa(r)<\wa(q)$. 
If $p \neq j$, $\epq \eij^{\mij+1} = \eij^{\mij+1} \epq \in \iprev$ since $\epq \in \iprev$. 
If $p = j$, $\ejq \eij^{\mij+1} = \eij^{\mij+1} \ejq - (\mij+1)\eij^{\mij}\eiq \in \iprev$ since $\ejq, \eiq \in \iprev$. 

\item $p=i$ and $q=j$ : Trivial from $\mij(\wa)=0$ and $\eij^{\mij(\wa)+1}\eij^{\mij+1} = \eij^{\mij+2} \in \iprev$. 

\vspace{2ex}
Hereafter we assume $p<q \leq j$ and $(p,q) \neq (i,j)$. 
\vspace{2ex}

\item $\{p,q\} \cap \{i,j\} = \varnothing$ : 
If $\mpq(\wa)=\mpq$ the proof is trivial 
since in this case $\epq^{\mpq(\wa)+1} \in \iprev$ and $\epq^{\mpq(\wa)+1}\eij^{\mij+1}=\eij^{\mij+1}\epq^{\mpq(\wa)+1}$. 

Consider the case $\mpq(\wa) \neq \mpq$. Then:

\begin{itemize}
\item $v(p)<v(q)$ must hold since otherwise $\mpq(\wa) = 0 = \mpq$, 
\item $q$ must be larger than $i$, since otherwise $\{\wa(r): r>q, \wa(p)<\wa(r)<\wa(q)\}=\{v(r): r>q, v(p)<v(r)<v(q)\}$ because $\wa$ and $v$ only differ at $i$-th and $j$-th positions, and 
\item exactly one of $v(i)$ and $v(j)$ must lie between $v(p)$ and $v(q)$ since otherwise $\{r > q : \wa(p)<\wa(r)<\wa(q)\} = \{r > q : v(p)<v(r)<v(q)\}$. 
\end{itemize}
Since $i<q<j$ and $\ell(vt_{ij})=\ell(v)+1$, the case $v(p)<v(i)<v(q)<v(j)$ cannot occur. So $v(i)<v(p)<v(j)<v(q)$. 
Then we have $p<i$ by the same reason. So we have $p<i<q<j$ and $v(i)<v(p)<v(j)<v(q)$. 

Here $\mpq(\wa)=\mpq-1$. 
Using the fact that there exists no $i<r<j$ with $v(i)<v(r)<v(j)$, 
we obtain $\miq-\mij = \#\{r > q : v(j)\leq v(r)<v(q)\} = \mpq-\mpj$. 

We have 
$
\epq^{\mpq}\eij^{\mij+1} \equiv \frac{(-1)^{\mij+1}}{(\mij+1)!}\epq^{\mpq}\eqj^{\mij+1}\eiq^{\mij+1} \pmod \iprev
$
by Lemma \ref{u3lem}(3) since $\eqj, \eiq^{\miq+1} \in \iprev$. 
Using $[e_{pq}, e_{qj}]=e_{pj}$ and $[e_{pq}, e_{pj}]=[e_{qj},e_{pj}]=0$ we see that 
the RHS is a linear combination of $\eqj^{\mij+1-\nu}\epq^{\mpq-\nu}\epj^{\nu}\eiq^{\mij+1}$ ($\nu \geq 0$). 
Thus it suffices to show that these elements are in $\iprev$ for each $\nu$. 
If $\nu > \mpj$ it is clear since $[\epj,\eiq]=0$ and $\epj^{\mpj+1} \in \iprev$. 
Otherwise, it suffices to show $\epq^{\mpq-\nu}\eiq^{\mij+1} \in \iprev$ since $[\epq,\epj]=0$. 
This follows from $\epq^{\mpq-\mpj}\eiq^{\mij+1}=\epq^{\miq-\mij}\eiq^{\mij+1} \in \iprev$, 
which can be deduced from $\epi, \eiq^{\miq+1} \in \iprev$ using Lemma \ref{u3lem}(1). 

\item $p=i$ : 
Since $i<q<j$, the case $v(i)<v(q)<v(j)$ cannot occur. 
If $v(q)<v(i)$, we have $\wa(q)<\wa(i)$ and thus $\miq(\wa)=0$. 
Therefore $\eiq^{\miq(\wa)+1}\eij^{\mij+1}=\eiq\eij^{\mij+1}=\eij^{\mij+1}\eiq \in \iprev$ since $\eiq \in \iprev$. 
If $v(q)>v(j)$, $\miq(\wa)=\miq-\mij-1$ since 
$\{r>q : \wa(i)<\wa(r)<\wa(q) \} = \{r>q : v(i)<v(r)<v(q) \} \smallsetminus (\{r>q : v(i)<v(r)<v(j) \} \cup \{j\}) = \{r>q : v(i)<v(r)<v(q) \} \smallsetminus (\{r>j : v(i)<v(r)<v(j) \} \cup \{j\})$,
so we want to show $\eiq^{\miq-\mij}\eij^{\mij+1} \in \iprev$. 
This follows from Lemma \ref{u3lem}(2) since $\eiq^{\miq+1}, \eqj \in \iprev$. 

\item $q=i$ : 
Here we have three cases to consider. 
If $v(p)<v(i)$, we have $\mpi(\wa)=\mpi+\mij+1$
since $\{r>i : \wa(p)<\wa(r)<\wa(i)\} = \{r>i : v(p)<v(r)<v(i) \} \cup \{r>i : v(i)<v(r)<v(j)\} \cup \{j\} = \{r>i : v(p)<v(r)<v(i) \} \cup \{r>j : v(i)<v(r)<v(j)\} \cup \{j\}$,
and so we want to show $\epi^{\mpi+\mij+2}\eij^{\mij+1} \in \iprev$. 
This follows from Lemma \ref{u3lem}(5) since $\epi^{\mpi+1}, \eij^{\mij+2} \in \iprev$. 
If $v(i)<v(p)<v(j)$, we have $\mpi(\wa)=\mpj$
since $\{r>i : \wa(p)<\wa(r)<\wa(i)\} = \{r>i : v(p)<v(r)<v(j)\} = \{r>j : v(p)<v(r)<v(j)\}$
and so we want to show $\epi^{\mpj+1}\eij^{\mij+1} \in \iprev$. 
This follows from Lemma \ref{u3lem}(6) since $\epi, \eij^{\mij+2}, \epj^{\mpj+1} \in \iprev$. 
Finally if $v(p)>v(j)$, we have $\wa(p)>\wa(i)$, $\mpi(\wa)=0$ and so we want to show $\epi\eij^{\mij+1} \in \iprev$. 
This follows from $\epi\eij^{\mij+1}=\eij^{\mij+1}\epi+(\mij+1)\eij^{\mij}\epj$ and $\epi, \epj \in \iprev$. 

\item $q=j$ :
This case consists of four subcases: 
\begin{itemize}
\item $p<i$ and $v(p)<v(i)$ : 
Here $\mpj(\wa)=\mpj-\mij$ since $\{r > j : \wa(p)<\wa(r)<\wa(j)\} = \{r > j : v(p)<v(r)<v(j)\} \smallsetminus \{r>j : v(i)<v(r)<v(j)\}$. 
So we want to show $\epj^{\mpj-\mij+1}\eij^{\mij+1} \in \iprev$. 
If there is no $r$ such that $i<r<j$ and $v(p)<v(r)<v(i)$, then $\mpi=\mpj-\mij$, 
and thus $\epj^{\mpj-\mij+1}\eij^{\mij+1}=\epj^{\mpi+1}\eij^{\mij+1} \in \iprev$ by Lemma \ref{u3lem}(1) since $\epi^{\mpi+1},\eij^{\mij+2} \in \iprev$. 
If there exists such $r$, take $r$ to be the largest among such ones. 
Then $\mpr=\mpj-\mrj$, since there exists no $s$ such that $r<s<j$ and $v(p)<v(s)<v(r)$. 
By $\eir, \erj^{\mrj+2} \in \iprev$ and Lemma \ref{u3lem}(4), 
we have $\epj^{\mpj-\mij+1}\eij^{\mij+1} \equiv \frac{1}{(\mij+1)!}\epj^{\mpj-\mij+1}\eir^{\mij+1}\erj^{\mij+1} \pmod \iprev$. 
Since the elements $\epr^{\mpr+1}=\epr^{\mpj-\mrj+1}$ and $\erj^{\mrj+2}$ are in $\iprev$ we see from Lemma \ref{u3lem}(1) that $\epj^{\mpj-\mij+1}\erj^{\mij+1} \in \iprev$. 
Thus $\epj^{\mpj-\mij+1}\eij^{\mij+1}\erj^{\mij+1} = \eij^{\mij+1}\epj^{\mpj-\mij+1}\erj^{\mij+1} \in \iprev$ and this shows the claim. 
\item $p<i$ and $v(p)>v(i)$ : 
Here $\mpj(\wa)=0$ since $\wa(p)>\wa(j)$. Thus $\epj^{\mpj(\wa)+1}\eij^{\mij+1}=\epj\eij^{\mij+1} \in \iprev$ by $\epi, \eij^{\mij+2} \in \iprev$ and Lemma \ref{u3lem}(1). 
\item $p>i$ and $v(p)<v(i)$ : 
Here $\mpj(\wa)=\mpj-\mij$ since $\{r > j : \wa(p)<\wa(r)<\wa(j)\} = \{r>j : v(p)<v(r)<v(j)\} \smallsetminus \{r>j : v(i)<v(r)<v(j)\}$. 
Thus $\epj^{\mpj(\wa)+1}\eij^{\mij+1}=\epj^{\mpj-\mij+1}\eij^{\mij+1} \in \iprev$ by $\eip, \epj^{\mpj+2} \in \iprev$ and Lemma \ref{u3lem}(1). 
\item $p>i$ and $v(p)>v(j)$ : 
Here $\mpj(\wa)=0$ since $\wa(p)>\wa(j)$. Thus $\epj^{\mpj(\wa)+1}\eij^{\mij+1}=\epj\eij^{\mij+2} \in \iprev$ since $\epj\eij^{\mij+2}=\eij^{\mij+2}\epj$ and $\epj \in \iprev$. 
\end{itemize}
\end{itemize}

Thus we checked $\epq^{\mpq(\wa)+1} x_a \in I^{(a-1)}$ for all $p<q$. This finishes the proof of Lemma \ref{mainlemma2}. \hfill $\Box$

\begin{rem}
It is clear from the definition that $m_{pr}(w) \leq m_{pq}(w)+m_{qr}(w)$ for any $p<q<r$. 
If $m_{pr}(w) = m_{pq}(w)+m_{qr}(w)$, then by Lemma \ref{u3lem}(1)
we have $e_{pr}^{m_{pr}(w)+1} \in \langle e_{pq}^{m_{pq}(w)+1}, e_{qr}^{m_{qr}(w)+1} \rangle$. 
Thus in fact the generators $e_{pr}^{m_{pr}(w)+1}$ such that
there exists some $q \in \{p+1, \ldots, r-1\}$ with $m_{pr}(w)=m_{pq}(w)+m_{qr}(w)$
are superfluous. 
\label{strict_triag}
\end{rem}

\section{Projectivity of KP modules}
\label{projectivity}
In this section we characterize KP modules by their projectivities in certain categories. 
This can be seen as an analog of Polo's theorem (\cite[Corollary 2.5]{P}, \cite[Theorem 3.1.10]{vdK}) for the case of KP modules. 

Let $\catc$ be the category of all weight modules. 
For $\Lmb \subset \ZZ^n$, 
let $\catc_{\Lmb}$ be the full subcategory of $\catc$ consists of all weight modules whose weights are in $\Lmb$. 
Note that if $|\Lmb| < \infty$ and $\Lmb'=\{\rho-\lmb : \lmb \in \Lmb \}$ ($\rho=(n-1, n-2, \ldots, 0)$), 
then $\catc_{\Lmb'} \cong \catc_{\Lmb}^{\mathrm{op}}$ by $M \mapsto M^{*} \otimes K_\rho$ (it is also true for infinite $\Lmb$ if we take $M^*$ to be the \textit{graded} dual $\bigoplus (M_\lmb)^*$ of $M$). 
\begin{lem}[{{cf. \cite[Lemma 3.1.1]{vdK}}}]
For finite $\Lmb \subset \ZZ^n$, $\catc_{\Lmb}$ has enough projectives (it is also true for infinite $\Lmb$ if we allow the weight spaces of a weight module to be infinite dimensional). 
\label{enoughlem}
\end{lem}
\begin{proof}
For $\lmb \in \Lmb$, let $P_\lmb=\ualg(\borel)/\langle h-\langle h,\lmb \rangle \rangle_{h \in \csa}$ (which is isomorphic to $\ualg(\nplus)$ as a $\ualg(\nplus)$-module, by PBW theorem) and let $P_\lmb^\Lmb$ be the largest quotient of $P_\lmb$ which is in $\catc_\Lmb$, 
i.e. $P_\lmb^\Lmb$ is the quotient of $P_\lmb$ by the submodule generated by all weight spaces $(P_\lmb)_\mu$ ($\mu \not\in \Lmb$). 
Then $P_\lmb^\Lmb$ is projective in $C_\Lmb$ since for $N \in C_\Lmb$, $\Hom(P_\lmb^\Lmb,N)=\Hom(P_\lmb,N)=N_\lmb$. 

For general $M \in \catc_\Lmb$, $P_M = \bigoplus_\lmb (P_\lmb^\Lmb)^{\oplus \dim M_\lmb}$ is a projective object in $\catc_\Lmb$ and there is a surjection $P_M \surj M$. This shows the lemma. 
\end{proof}
Note that, if $\lmb \in \Lmb$, $P_\lmb^\Lmb$ has the maximum proper submodule $\bigoplus_{\mu \neq \lmb} (P_\lmb^\Lmb)_\mu$; therefore the head of $P_\lmb^\Lmb$ is $K_\lmb$, and thus $P_\lmb^\Lmb$ is the projective cover of $K_\lmb$ in $\catc_\Lmb$. 

We introduce two orderings (other than dominance order) on $\ZZ^n$ as follows. 
For two permutations $w, v \in S_\infty$, 
we write $w \lex\leq v$ if $w=v$ or there exists an $i \in \NN$ such that $w(j)=v(j)$ for all $j<i$ and $w(i)<v(i)$. 
Likewise, we write $w \rlex\leq v$ if $w=v$ or there exists an $i \in \NN$ such that $w(j)=v(j)$ for all $j>i$ and $w(i)<v(i)$. 
For $\lmb=(\lmb_1, \ldots, \lmb_n) \in \ZZ^n$, define $|\lmb|=\sum \lmb_i$. 
If $\lmb, \mu \in \nonneg^n$ and $w=\perm(\lmb), v=\perm(\mu)$, 
we write $\lmb \geq \mu$ if $|\lmb|=|\mu|$ and $w^{-1} \lex \leq v^{-1}$. 
For general $\lmb$ and $\mu$ in $\ZZ^n$, take $k$ so that $\lmb+k\one$ and $\mu+k\one$ are in $\nonneg^n$, 
and define $\lmb \geq \mu \iff \lmb+k\one \geq \mu+k\one$. 
Note that this definition does not depend on the choice of $k$ 
since $\perm(\lmb)^{-1} \lex\leq \perm(\mu)^{-1} \iff \perm(\lmb+\one)^{-1} \lex\leq \perm(\mu+\one)^{-1}$ for $\lmb,\mu \in \nonneg^n$. 
We define the other ordering $\geq'$ in the same way, except that we use $\rlex\leq$ instead of $\lex\leq$. 
We prepare the following two lemmas about these orderings:
\begin{lem}
For $\lmb, \mu \in \ZZ^n$, $\lmb \geq \mu$ if and only if $\rho-\lmb \geq' \rho-\mu$. 
\label{twoorder}
\end{lem}
\begin{proof}
We may assume $|\lmb|=|\mu|$. 
We only need to prove the ``only if'' direction since the other implication follows by exchanging $\lmb$ and $\mu$. 
Take integers $L$ and $M$ so that $\lmb+L\one, \mu+L\one, \rho-\lmb+M\one, \rho-\mu+M\one \in \nonneg^n$. 
Let $w=\perm(\lmb+L\one), v=\perm(\mu+L\one), w'=\perm(\rho-\lmb+M\one)$ and $v'=\perm(\rho-\mu+M\one)$. 
Then $w, v, w', v' \in S_\infty^{(n)} \cap S_N$, and these permutations are related by $w'(i)=N+1-w(i), v'(i)=N+1-v(i)$ for $i=1, \ldots, n$, where $N=n+L+M$. 
Thus, for $p \in \{1, \ldots, N\}$, $w'^{-1}(p) \leq n$ (resp. $v'^{-1}(p) \leq n$) if and only if $w^{-1}(N+1-p) \leq n$ (resp. $v^{-1}(N+1-p) \leq n$), and 
$w'^{-1}(p)=
\begin{cases}
w^{-1}(N+1-p) & (w^{-1}(p) \leq n) \\
n+N+1-w^{-1}(N+1-p) & (w^{-1}(p) > n)
\end{cases}$
and
$v'^{-1}(p)=
\begin{cases}
v^{-1}(N+1-p) & (v^{-1}(p) \leq n) \\
n+N+1-v^{-1}(N+1-p) & (v^{-1}(p) > n)
\end{cases}$. 

Now assume $w \lex\leq v$. if $w=v$ we have nothing to prove so we assume that there is an $i$ such that $w^{-1}(1)=v^{-1}(1), \ldots, w^{-1}(i-1)=v^{-1}(i-1), w^{-1}(i)<v^{-1}(i)$. 
By the above description of $w'$ and $v'$ it is clear that $w'^{-1}(j) = v'^{-1}(j)$ for $j > N+1-i$. 
We show $w'^{-1}(N+1-i) < v'^{-1}(N+1-i)$. 
If $w^{-1}(i)<v^{-1}(i) \leq n$ we have $w'^{-1}(N+1-i)=w^{-1}(i)<v^{-1}(i)=v'^{-1}(N+1-i)$. 
If $w^{-1}(1) \leq n < v^{-1}(i)$ we have $w'^{-1}(N+1-i) \leq n < v'^{-1}(N+1-i)$.
The case $n < w^{-1}(i)<v^{-1}(i)$ cannot occur, 
since in such case $w^{-1}(i)=n+1+\#\{j<i : w^{-1}(j)>n\}$, $v^{-1}(i)=n+1+\#\{j<i : v^{-1}(j)>n\}$ and $\{j<i : w^{-1}(j)>n\}=\{j<i : v^{-1}(j)>n\}$. 
Thus we have checked $w'^{-1}(N+1-i) < v'^{-1}(N+1-i)$ and thus $w'^{-1} \rlex\leq v'^{-1}$. This shows the lemma. 
\end{proof}

\begin{lem}
\label{linord}
For any $\lmb \in \ZZ^n$, the set $\{\nu : \nu \leq \lmb\}$ is finite and linearly ordered by $\leq$. 
\end{lem}
\begin{proof}
Linear-orderedness is clear from the definition of $\leq$. 
We claim that if $\lmb, \mu \in \ZZ^n$, $|\lmb| = |\mu|$ and $\min_i \lmb_i > \min_i \mu_i$ then $\lmb < \mu$
(this shows the lemma since there exists only finitely many $\nu \in \ZZ^n$
such that $|\nu|=|\lmb|$ and $\min_i \nu_i \geq \min_i \lmb_i$). 
We may assume that $\lmb, \mu \in \nonneg^n$. Let $m=\min_i \mu_i$. 
Then $w=\perm(\lmb)$ and $v=\perm(\mu)$ satisfy
$w^{-1}(1)=v^{-1}(1)=n+1, \ldots, w^{-1}(m)=v^{-1}(m)=n+m$ and $w^{-1}(m+1) > n \geq v^{-1}(m+1)$. 
Thus $w^{-1} \lex\geq v^{-1}$. 
\end{proof}

We define $\catc_{\leq \lmb}=\catc_{\{\nu : \nu \leq \lmb\}}$, $\catc_{< \lmb}=\catc_{\{\nu : \nu < \lmb\}}$
and $\catc_{\leq' \lmb}=\catc_{\{\nu : \nu \leq' \lmb\}}$. 
The main result of this section is the following proposition: 
\begin{prop}
For $\lmb \in \ZZ^n$, the modules $\smod_\lmb$ and $\smod_{\rho-\lmb}^* \otimes K_\rho$ are in $\catc_{ \leq \lmb }$, 
Moreover, $\smod_\lmb$ is projective and $\smod_{\rho-\lmb}^* \otimes K_\rho$ is injective in $\catc_{ \leq \lmb }$. 
\label{projinj}
\end{prop}
Note that, by the remark before Lemma \ref{enoughlem}, the last claim is equivalent to the claim that 
$\smod_{\rho-\lmb}$ is projective in $\catc_{\{\rho-\nu : \nu \leq \lmb\}}=\catc_{\leq' \rho-\lmb}$. 
Moreover, since the head of $\smod_\lmb$ is $K_\lmb$, 
this proposition claims that $\smod_\lmb$ is the projective cover of $K_\lmb$ in both $\catc_{\leq \lmb}$ and $\catc_{\leq' \lmb}$, 
i.e., $\smod_\lmb \cong P_\lmb^{\leq \lmb} \cong P_\lmb^{\leq' \lmb}$
(we write $P_\mu^{\leq \lmb}$ and $P_\mu^{\leq' \lmb}$ for $P_\mu^{\{\nu : \nu \leq \lmb\}}$ and $P_\mu^{\{\nu : \nu \leq' \lmb\}}$ respectively). 
We also remark that the proposition implies the same statement for $\leq'$ instead of $\leq$, by Lemma \ref{twoorder}. 

To prove Proposition \ref{projinj}, 
we have to prove the following four facts: for every $\lmb, \mu \in \ZZ^n$, 
\begin{enumerate}
\item $(\smod_{\lmb})_\mu \neq 0$ implies $\lmb \ge \mu$, 
\item $(\smod_{\rho-\lmb}^*\otimes K_\rho)_\mu \neq 0$ (which is equivalent to $(\smod_{\rho-\lmb})_{\rho-\mu} \neq 0$) implies $\lmb \ge \mu$, 
\item $\Ext^1(\smod_{\lmb}, K_\mu) \neq 0$ implies $\lmb < \mu$
(here $\Ext^1$ is taken in either $\catc$ or $\catc_{\leq \lmb}$, which does not matter since $\catc_{\leq \lmb}$ is closed under extension), and 
\item $\Ext^1(\smod_{\rho-\lmb}, K_{\rho-\mu}) \neq 0$ implies $\lmb < \mu$. 
\end{enumerate}

Before starting the proof, first let us make a observation on the weights of $\smod_w$ ($w \in S_\infty^{(n)}$). 
Let $l_j(w) = \#\{i : i<j, w(i)>w(j)\}$ as in the definition of KP modules. 
Since $\smod_w$ is a submodule of $\bigotimes_{j \geq 1} \bigwedge^{l_j(w)} K^{j-1}$, 
any weight of $\smod_w$ is a weight of $\bigotimes_{j \geq 1} \bigwedge^{l_j(w)} K^{j-1}$. 
The weights of the latter space can be understood as follows. A \textit{$w$-pattern} (terminology only for here) is a sequence of sets $(I_1, I_2, \ldots)$ such that $I_j \subset \{1, \ldots, j-1\}$ and $|I_j|=l_j(w)$. Define the \textit{weight} $(\mu_1, \mu_2, \ldots)$ of a $w$-pattern $(I_1, I_2,  \ldots)$ by $\mu_i=\#\{j : i \in I_j\}$. Then it is easy to see that $\mu$ is a weight of $\bigotimes_{j \geq 1} \bigwedge^{l_j(w)} K^{j-1}$ if and only if it is the weight of some $w$-pattern. 

Let us now prove (1) and (2) above. 

\vspace{1ex}
(1): 
We may assume that $\lmb$ and $\mu$ are in $\nonneg^n$, since $(\smod_\lmb)_\mu \neq 0 \iff (\smod_{\lmb+k\one})_{\mu+k\one} \neq 0$ for any $\lmb, \mu \in \ZZ^n$ and any $k \in \ZZ$. 
Let $w=\perm(\lmb)$ and $v=\perm(\mu)$. 
We prove a stronger statement: if $\mu$ is the weight of some $w$-pattern $(I_1, I_2, \ldots)$ then $\lmb \geq \mu$. 

We first show $w^{-1}(1) \leq v^{-1}(1)$. 
Let $i=w^{-1}(1)$. 
Since $w(1), \ldots, w(i-1) > w(i)$ we have $l_i(w)=i-1$, and thus $I_{i}=\{1, \ldots, i-1\}$. 
Thus $\mu_1, \ldots, \mu_{i-1} \geq 1$. 
Since $v^{-1}(1)=\min\{j : \mu_j = 0\}$, this shows $w^{-1}(1) \leq v^{-1}(1)$. 

Now consider the case $w^{-1}(1)=v^{-1}(1)$. 
In this case we have $\mu_i=0$, i.e. none of the sets $I_j$ contains $i$. 
Define $\sigma_i: \NN \smallsetminus \{i\} \mor \NN$ by $\sigma_i(i')=\begin{cases}i' & (i'<i) \\ i'-1 & (i'>i) \end{cases}$, and consider a new sequence of sets $I'=(\sigma_i(I_1), \ldots, \sigma_i(I_{i-1}), \sigma_i(I_{i+1}), \sigma_i(I_{i+2}), \ldots)$. 
It is easy to check that $I'$ is a $w'$-pattern with weight $\inv(v')$, where $w'=[w(1)-1 \; \cdots \; w(i-1)-1 \; w(i+1)-1 \; w(i+2)-1 \cdots]$ and $v'=[v(1)-1 \; \cdots \; v(i-1)-1 \; v(i+1)-1 \; v(i+2)-1 \cdots]$. An inductive argument shows that $w'^{-1} \lex\leq v'^{-1}$. This shows $w^{-1} \lex\leq v^{-1}$. \hfill $\Box$

\vspace{1ex}
(2):
We may assume $\lmb, \mu \in \nonneg^n$ as before. Let $w=\perm(\lmb)$ and $v=\perm(\mu)$. 
We prove a stronger statement: if $\mu$ is the weight of some $w$-pattern $(I_1, I_2, \ldots)$ then $\rho-\lmb \geq \rho-\mu$, or equivalently (by Lemma \ref{twoorder}), $\lmb \geq' \mu$. 
Take $N$ so that $w, v \in S_N$. Note that $I_{N+1}=I_{N+2}=\cdots=\varnothing$ since $l_w(N+1)=l_w(N+2)=\cdots=0$. 

We first show $w^{-1}(N) \leq v^{-1}(N)$. 
Let $i=w^{-1}(N)$.
Then we have $l_i(w)=0$ and thus $I_i=\varnothing$. 
Thus for $j<i$, we have $j \not\in I_1, \ldots I_j$ and $j \not\in I_i$. Thus $\mu_j \leq N-j-1$. 
Since $v^{-1}(N)=\min\{i : \mu_i = N-i\}$ this shows $v^{-1}(N) \geq w^{-1}(N)$. 

Now consider the case $w^{-1}(N)=v^{-1}(N)$. Then $\mu_i=N-i$. 
Since $i \not\in I_1, \ldots, I_i$ we must have $i\in I_{i+1}, \ldots I_N$. 
It is easy to see that  $I'=(\sigma_i(I_1), \ldots, \sigma_i(I_{i-1}), \sigma_i(I_{i+1} \smallsetminus \{i\}), \ldots, \sigma_i(I_N \smallsetminus \{i\}), \varnothing, \varnothing, \ldots)$ is a $w'$-pattern with weight $\inv(v')$ where $w'=[w(1) \, \cdots \, w(i-1) \, w(i+1) \, \cdots \, w(N)]$ and  $v'=[v(1) \, \cdots \, v(i-1) \, v(i+1) \, \cdots \, v(N)]$. 
An inductive argument shows $w'^{-1} \rlex\leq v'^{-1}$. This shows $w^{-1} \rlex\leq v^{-1}$.  \hfill $\Box$

\vspace{1ex}
For (3) and (4), we need the following observation. 
By Theorem \ref{mainthm}, for any $w \in S_\infty^{(n)}$ there is a projective resolution of $\smod_w$ in $\catc$ of the form $\cdots \mor P_1 \mor P_0 \mor \smod_w \mor 0$
with $P_0=P_{\inv(w)}$ and $P_1=\bigoplus_{p<q} P_{\inv(w)+(\mpq(w)+1)(\epsilon_p-\epsilon_q)}$. 
Here by Remark \ref{strict_triag}, we can in fact replace $P_1$ by a smaller module: sum over all $p<q$ such that \[(*): \text{there does not exist $p<r<q$ with $\mpq(w)=m_{pr}(w)+m_{rq}(w)$}.\]
In particular, $\Ext^1(\smod_w, K_\mu)=0$ unless $\mu=\inv(w)+(\mpq(w)+1)(\epsilon_p-\epsilon_q)$ for some $p<q$ satisfying the property $(*)$ above. 

\vspace{1ex}
(3): We may assume that $\lmb, \mu \in \nonneg^n$, since $\lmb < \mu \iff \lmb+k\one < \mu+k\one$ and $\Ext^1(\smod_\lmb, K_\mu) \neq 0 \iff \Ext^1(\smod_{\lmb+k\one}, K_{\mu+k\one}) \neq 0$ for any $\lmb, \mu \in \ZZ^n$ and any $k \in \ZZ$. 

Let $w=\perm(\lmb)$ and $v=\perm(\mu)$. 
By the remark above, we have $\mu=\lmb+(\mpq(w)+1)(\epsilon_p-\epsilon_q)$ for some $p<q$ (and therefore $w \neq v$). 
We first show $w^{-1}(1) \geq v^{-1}(1)$. 
Let $i=w^{-1}(1)$. 
If $i < v^{-1}(1)$, then $\mu_i>0$ while $\lmb_i=0$, 
and so $p=i$. 
But then $m_{pq}(w)=\inv(w)_q=\lmb_q$ and so we have $\mu_q=-1$, which contradicts to $\mu \in \nonneg^n$. 
Therefore $i \geq v^{-1}(1)$. 

If $i=v^{-1}(1)$, then $\lmb_i=\mu_i=0$, and so $p, q \neq i$. 
Therefore $\lmb'=(\lmb_1-1, \ldots, \lmb_{i-1}-1, \lmb_{i+1}, \lmb_{i+2}, \ldots)$
and $\mu'=(\mu_1-1, \ldots, \mu_{i-1}-1, \mu_{i+1}, \mu_{i+2}, \ldots)$ 
satisfy $\mu'=\lmb'+(m_{pq}(w)+1)(\epsilon_{p'}-\epsilon_{q'})$ for $p'=\sigma_i(p)$, $q'=\sigma_i(q)$. 
Moreover, $m_{pq}(w) = m_{p'q'}(w')$, where $w'=[w(1)-1 \; \cdots \; w(i-1)-1 \; w(i+1)-1 \; w(i+2)-1 \; \cdots]$. 
Thus an inductive argument shows $w'^{-1} \lex\geq v'^{-1}$ where $v'=[v(1)-1 \; \cdots \; v(i-1)-1 \; v(i+1)-1 \; v(i+2)-1 \; \cdots]$. This shows $w^{-1} \lex\geq v^{-1}$. \hfill $\Box$

\vspace{1ex}
(4): We assume $\Ext^1(\smod_{\lmb}, K_{\mu}) \neq 0$ and prove $\rho-\lmb < \rho-\mu$, or equivalently, $\lmb <' \mu$. 
We may assume that $\lmb, \mu \in \nonneg^n$ as before. 
Let $w=\perm(\lmb)$, $v=\perm(\mu)$. 
Take $N$ so that $w, v \in S_N$. 
We have $\mu=\lmb+(\mpq(w)+1)(\epsilon_p-\epsilon_q)$ for some $p<q$ as before, 
with the property $(*)$ remarked above. 
We first show $w^{-1}(N) \geq v^{-1}(N)$. 

Assume $w^{-1}(N)<v^{-1}(N)$. 
Then $\lmb_{w^{-1}(N)}=N-w^{-1}(N)$ while $\mu_{w^{-1}(N)}<N-{w^{-1}(N)}$ and so $q=w^{-1}(N)$. 

We first claim that there does not exist $r$ such that $p<r<q$ and $w(p)<w(r)$. 
Suppose such $r$ exists. Take $r$ to be the largest among such. 
By the property $(*)$ we have $m_{pq}(w)<m_{pr}(w)+m_{rq}(w)$. 
This means that there is a column index $1 \leq j \leq N$
such that $(p,j), (q,j) \in I(w), (r,j) \not\in I(w)$ or $(p,j), (q,j) \not\in I(w), (r,j) \in I(w)$, 
since other types of column contribute to LHS and RHS by the same value. 
We see that neither of these cases cannot occur as follows. 
\begin{itemize}
\item Assume the former case. Then $(p,j) \in I(w)$ implies $w(j)<w(p)<w(r)$ and $(q,j) \in I(w)$ implies $j>q>r$. These shows $(r,j) \in I(w)$. Contradiction. 
\item Assume the latter case. $w(q) = N > w(j)$ and $(q,j) \not\in I(w)$ implies $j<q$. Also, $(r,j) \in I(w)$ implies $j>r>p$, and this together with $(p,j) \not\in I(w)$ shows $w(p)<w(j)$. Thus $j$ satisfies $p<j<q$, $w(p)<w(j)$ and $j>r$. This contradicts to the choice of $r$. 
\end{itemize}

Since there does not exist $r$ such that $p<r<q$ and $w(p)<w(r)$, 
we see that $m_{pq}(w) = \#\{r > q : w(p)<w(r)<w(q)\} = \#\{r > q : w(p)<w(r)\} = N-w(p)-1-\#\{r<q : w(p)<w(r)\} =  N-w(p)-1-\#\{r<p : w(p)<w(r)\}$. 
From this and $\lmb_p=\inv(w)_p=\#\{r>p : w(r)<w(p)\}=w(p)-1-\#\{r<p : w(r)<w(p)\}$, 
we see $\mu_p=\lmb_p+m_{pq}(w)+1=N-p$. 
This means $v^{-1}(N)=\min\{p' : \mu_{p'} = N-p'\} \leq p < q = w^{-1}(N)$. This contradicts to the assumption and thus we see $w^{-1}(N) \geq v^{-1}(N)$. 

If $w^{-1}(N)=v^{-1}(N)$, then $p \neq w^{-1}(N) \neq q$ as before, and we can inductively argue in the same way as in (3). \hfill $\Box$

\section{Vanishing of higher extensions}
\label{higherext}
In this section, we prove an analogue of ``Strong form of Polo's theorem'' (\cite[Theorem 3.2.2]{vdK}) for KP modules: i.e. the vanishing of higher extensions $\Ext^i(\smod_\lmb, \smod_\mu^* \otimes K_\rho)$ ($i \geq 1$).
To prove it we need the following lemma: 
\begin{lem}
For any $\lmb \geq \lmb'$, $M, N \in \catc_{\leq \lmb'}$ and $i \geq 0$, $\Ext^i_{\leq \lmb}(M,N) \cong \Ext^i_{\leq \lmb'}(M,N)$. 
Here $\Ext^i_{\leq \lmb}$ is short for $\Ext^i_{\catc_{\leq \lmb}}$. 
\label{sugoilem}
\end{lem}
\begin{proof}
It is enough (by Lemma \ref{linord}) to prove $\Ext^i_{\leq \lmb}(M,N)=\Ext^i_{\catc_{<\lmb}}(M,N)$ for $M,N \in \catc_{<\lmb}$. 
Take a projective resolution $\cdots \mor P_1 \mor P_0 \mor M \mor 0$ such that each $P_i$ is a direct sum of some modules $P_\mu^{\leq \lmb}$ with $\mu \leq \lmb$
(in fact, the only indecomposable projectives in $\catc_{\leq \lmb}$ are $P_\mu^{\leq \lmb}$, so this condition is superfluous). 
For $L \in \catc_{\leq \lmb}$, let $\overline{L}$ be the largest quotient of $L$ which is in $\catc_{< \lmb}$, 
i.e., $\overline{L}$ is the quotient of $L$ by the submodule generated by the weight space $L_\lmb$ of weight $\lmb$. 
Note that if $P_i=P_\mu^{\leq \lmb} \oplus P_\nu^{\leq \lmb} \oplus \cdots$, then $\overline{P_i}=P_\mu^{< \lmb} \oplus P_\nu^{< \lmb} \cdots$ 
where $P_\mu^{< \lmb}, P_\nu^{< \lmb}, \ldots$ are the largest quotients of $P_\mu, P_\nu, \ldots$ which are in $\catc_{< \lmb}$. 
We are done if we show that $\cdots \mor \overline{P_1} \mor \overline{P_0} \mor M \mor 0$ is a projective resolution of $M$, 
since $\Hom(\overline{P_i},N)=\Hom(P_i,N)$. 
It is clear that each $\overline{P_i}$ is projective. 
Let $\Ker_i$ be the kernel of $P_i \surj \overline{P_i}$. 
Since $\cdots \mor P_1 \mor P_0 \mor M \mor 0$ is exact, 
the exactness of $\cdots \mor \overline{P_1} \mor \overline{P_0} \mor M \mor 0$
is equivalent to that of $\cdots \mor \Ker_1 \mor \Ker_0 \mor 0$. 

For any $\mu, \nu \leq \lmb$, we have a linear map $(P_\mu^{\leq \lmb})_\lmb \otimes (P_\lmb^{\leq \lmb})_\nu \mor (P_\mu^{\leq \lmb})_\nu$
 defined by $x u_\mu \otimes y u_\lmb \mapsto yx u_\mu$ for $x \in \ualg(\nplus)_{\lmb-\mu}$ and $y \in \ualg(\nplus)_{\nu-\lmb}$ 
where $u_\mu$ is the image of $1 \otimes 1 \in \ualg(\nplus) \otimes K_\mu \cong P_\mu \surj P_\mu^{\leq \lmb}$
(this definition does not depend on the choice of $y$ since the submodule of $P_\mu^{\leq \lmb}$ generated by $x u_\mu$ is a quotient of $P_\lmb^{\leq \lmb}$ by definition). 
This map induces a surjection from $(P_\mu^{\leq \lmb})_\lmb \otimes (P_\lmb^{\leq \lmb})_\nu$ to $\Ker((P_\mu^{\leq \lmb})_\nu \surj (P_\mu^{< \lmb})_\nu)$, since the kernel is, by definition, generated by $(P_\mu^{\leq \lmb})_\lmb$ as a $\ualg(\borel)$-module. 
We claim that this surjection is in fact an isomorphism, for any $\lmb$ and any $\mu, \nu \leq \lmb$. 
Note that the claim implies the lemma: 
if we show this we have $(P_i)_\lmb \otimes (P_\lmb^{\leq \lmb})_\nu \cong (\Ker_i)_\nu$ for each $i$, 
and thus the exactness of $\cdots \mor \Ker_1 \mor \Ker_0 \mor 0$ follows from that of $\cdots \mor (P_1)_\lmb \mor (P_0)_\lmb \mor 0$. 

For $\lmb \in \ZZ^n$ and $\mu, \nu \leq \lmb$, we have a quotient filtration $(P_\mu^{\leq \lmb})_\nu=(P_\mu^{\leq \kap^{(r)}})_\nu \surj (P_\mu^{\leq \kap^{(r-1)}})_\nu \surj \cdots \surj (P_\mu^{\leq \kap^{(1)}})_\nu \surj 0$, where $\kap^{(1)} < \cdots < \kap^{(r)} = \lmb$ are the elements of $\ZZ^n$ less than or equal to $\lmb$. 
By the argument above, the subquotient $\Ker((P_\mu^{\leq \kap^{(i)}})_\nu \surj (P_\mu^{\leq \kap^{(i-1)}})_\nu)$ of this filtration is a quotient of $(P_\mu^{\leq \kap^{(i)}})_{\kap^{(i)}} \otimes (P_{\kap^{(i)}}^{\leq \kap^{(i)}})_\nu$. 
Thus $\dim (P_\mu^{\leq \lmb})_\nu \leq \sum_{\kap \leq \lmb} \dim ((P_\mu^{\leq \kap})_{\kap} \otimes (P_{\kap}^{\leq \kap})_\nu)$. 
If we show that the equality holds, then the desired isomorphism $(P_\mu^{\leq \kap^{(i)}})_{\kap^{(i)}} \otimes (P_{\kap^{(i)}}^{\leq \kap^{(i)}})_\nu \cong \Ker((P_\mu^{\leq \kap^{(i)}})_\nu \surj (P_\mu^{\leq \kap^{(i-1)}})_\nu)$ follows for all $i$: in particular, proving the equality for a sufficiently large $\lmb$ (with respect to the ordering $\leq$) is enough for the proof of Lemma \ref{sugoilem}. 

We know $(P_\kap^{\leq \kap})_\nu \cong (\smod_\kap)_\nu$ by Proposition \ref{projinj}. Now consider $(P_\mu^{\leq \kap})_\kap$. Since $P_\mu^{\leq \kap}$ is the quotient of $P_\mu$ by the submodule generated by all weight spaces $(P_\mu)_\sigma$ ($\sigma \not\leq \kap$), we see that
\[
(P_\mu^{\leq \kap})_\kap \cong \ualg(\nplus)_{\kap-\mu}/\mathrm{Span}_K \{ xy : \text{$x \in \ualg(\nplus)_{\kap-\sigma}$, $y \in \ualg(\nplus)_{\sigma-\mu}$ for some $\sigma \not\leq \kap$} \}.
\]
The algebra antiautomorphism on $\ualg(\nplus)$ given by $X \mapsto X$ ($X \in \nplus$) induces an isomorphism between this space and
\begin{align*}
&\ualg(\nplus)_{\kap-\mu}/\mathrm{Span}_K \{ yx : \text{$x \in \ualg(\nplus)_{\kap-\sigma}$, $y \in \ualg(\nplus)_{\sigma-\mu}$ for some $\sigma \not\leq \kap$} \} \\
&= \ualg(\nplus)_{\kap-\mu}/\mathrm{Span}_K \{ yx : \text{$x \in \ualg(\nplus)_{\kap-\sigma}$, $y \in \ualg(\nplus)_{\sigma-\mu}$ for some $\sigma$ s.t. $\rho-\sigma \not\leq' \rho-\kap$} \}.
\end{align*}
By the same argument as above we see that this is isomorphic to $(P_{\rho-\kap}^{\leq' \rho-\kap})_{\rho-\mu}$. By Proposition \ref{projinj} (and Lemma \ref{twoorder}) we see $(P_{\rho-\kap}^{\leq' \rho-\kap})_{\rho-\mu} \cong (\smod_{\rho-\kap})_{\rho-\mu}$. Thus, after all, we see that $(P_\mu^{\leq \kap})_\kap \cong (\smod_{\rho-\kap})_{\rho-\mu}$. 

Since as we have seen above $(P_\mu^{\leq \kap})_\kap \cong (\smod_{\rho-\kap})_{\rho-\mu}$ and $(P_\kap^{\leq \kap})_\nu \cong (\smod_\kap)_\nu$, we see that $\dim ((P_\mu^{\leq \kap})_{\kap} \otimes (P_{\kap}^{\leq \kap})_\nu)$ is equal to the coefficient of $x^{\rho-\mu} y^\nu$ in $\schub_{\rho-\kap}(x) \schub_\kap(y)$. 
Also, $\dim(P_\mu^{\leq \lmb})_\nu =\dim \ualg(\nplus)_{\nu-\mu}$ if $\lmb$ is sufficiently large with respect to $\leq$. Thus the proof of Lemma \ref{sugoilem} is now reduced to the following elementary lemma: 

\begin{lem}
For $\mu, \nu \in \ZZ^n$, $\dim \ualg(\nplus)_{\nu-\mu}$ is equal to the coefficient of $x^{\rho-\mu}y^\nu$ in $\sum_{\kap \in \ZZ^n} \schub_{\rho-\kap}(x)\schub_\kap(y)$. 
\label{sansuu}
\end{lem}
Let us prove this lemma. We use the following result from \cite{PS}: 
\begin{lem}[{{\cite[Lemma 6.2 and Corollary 9.2, reformulated]{PS}}}]
For a positive integer $N$, define a bilinear form $\langle \cdot,\cdot \rangle$ on $\ZZ[x_1^{\pm 1}, \ldots, x_N^{\pm 1}]$ by $\langle x^\alpha,x^\beta \rangle=\delta_{\alpha\beta}$. 
Then for $w, v \in S_N$, $\langle \schub_w, \schub_{w_0v}(x_1^{-1},\ldots,x_N^{-1})\prod_{1 \leq i<j \leq N}(x_i-x_j) \rangle=\delta_{wv}$, 
where $w_0 = [N \; N-1 \; \cdots \; 1] \in S_N$. 
\end{lem}
We slightly modify this lemma into a form which is more suitable for our use: 
\begin{lem}
If we define a bilinear form $\langle \cdot,\cdot \rangle$ on $\ZZ[x_1^{\pm 1}, \ldots, x_n^{\pm 1}]$ by $\langle x^\alpha,x^\beta \rangle=\delta_{\alpha\beta}$, 
then for $\lmb, \mu \in \ZZ^n$, $\langle \schub_\lmb, \schub_{\rho-\mu}(x_1^{-1}, \ldots, x_n^{-1})\prod_{1 \leq i<j \leq n}(x_i-x_j) \rangle = \delta_{\lmb\mu}$.  
\label{dualbasislem}
\end{lem}
\begin{proof}
We may assume that $\lmb, \mu \in \nonneg^n$. 
Let $w=\perm(\lmb), v=\perm(\mu)$. Take $N$ so that $w, v \in S_N$. 
Then by the previous lemma, 
we have 
\[
\langle \schub_w, \schub_{w_0v}(x_1^{-1}, \ldots, x_N^{-1})\prod_{1 \leq i<j \leq N}(x_i-x_j) \rangle = \delta_{wv}=\delta_{\lmb\mu} \quad \cdots (*)
\]
where $w_0 = [N \; N-1 \; \cdots \; 1] \in S_N$. 

Since $\prod_{1 \leq i<j \leq N}(x_i-x_j) =  \prod_{i\leq n < j}(x_i-x_j) \cdot \prod_{i<j \leq n}(x_i-x_j) \cdot \prod_{n < i < j}(x_i-x_j)$, it can be seen that 
\begin{align*}
\prod_{1 \leq i<j \leq N}(x_i-x_j) &\equiv (x_1 \cdots x_n)^{N-n} \cdot \prod_{i<j \leq n}(x_i-x_j) \cdot \prod_{n < i < j}(x_i-x_j) \\
&= (x_1 \cdots x_n)^{N-n} \cdot \prod_{i<j \leq n}(x_i-x_j) \cdot (x_{n+1}^{N-n-1} x_{n+2}^{N-n-2}\cdots x_{N-1} + R)
\end{align*}
modulo terms whose total degree in variables $x_{n+1}, \ldots, x_N$ is strictly larger than $T=\binom{N-n}{2}$, and $R$ is some polynomial in $x_{n+1}, \ldots, x_N$ with degree $T$ and without monomial $x_{n+1}^{N-n-1} x_{n+2}^{N-n-2}\cdots x_{N-1}$. 

Let $f$ be the sum of all terms in $\schub_{w_0v}$ whose degree in $x_{n+1}, \ldots, x_N$ is equal to $T$. 
Note that, since $\schub_{w_0v}$ is a linear combination of monomials $x_1^{a_1} \cdots x_n^{a_n}$ ($0 \leq a_i \leq N-i$), the degree in $x_{n+1}, \ldots, x_N$ of its terms are always at most $T$: that is, $\schub_{w_0v}=f+(\text{terms with degree $<T$ in variables $x_{n+1}, \ldots, x_N$})$. Also note $f \in x_{n+1}^{N-n-1} \cdots x_{N-1}\ZZ[x_1, \ldots, x_n]$ by the same reason. We claim $f=(x_1 \cdots x_n)^{N-n}x_{n+1}^{N-n-1} \cdots x_{N-1}\schub_{\rho-\mu}$. 

Let $w_{n,N} = [1 \; \cdots \; n \; N \; N-1 \; \cdots \; n+1] \in S_N$. Note that $\inv(w_{n,N}w_0v)=\rho-\mu+(N-n)\one$ and thus $\schub_{w_{n,N}w_0v}= (x_1 \cdots x_n)^{N-n}\schub_{\rho-\mu}$. 
We have $\schub_{w_{n,N}w_0v} = \der_{w_{n,N}} \schub_{w_0v}$, where $\der_{w_{n,N}} = (\der_{n+1} \der_{n+2} \cdots \der_{N-1}) \cdot  (\der_{n+2} \cdots \der_{N-1}) \cdot \cdots \cdot \der_{N-1}$. Since the operators $\der_i$ ($n+1 \leq i \leq N-1$) lower the degree in variables $x_{n+1}, \ldots, x_N$ by one, $\der_{w_{n,N}}$ annihilates $\schub_{w_0v}-f$. Thus $\schub_{w_{n,N}w_0v} = \der_{w_{n,N}} f$. Since $f \in x_{n+1}^{N-n-1} \cdots x_{N-1}\ZZ[x_1, \ldots, x_n]$ and $\der_{w_{n,N}} x_{n+1}^{N-n-1} \cdots x_{N-1} = 1$ we see that $\der_{w_{n,N}} f = f/(x_{n+1}^{N-n-1} \cdots x_{N-1})$. Thus $f=x_{n+1}^{N-n-1} \cdots x_{N-1}\schub_{w_{n,N}w_0v} =(x_1 \cdots x_n)^{N-n}x_{n+1}^{N-n-1} \cdots x_{N-1}\schub_{\rho-\mu}$. This shows the claim above. 

We have seen that 
\[
\prod_{1 \leq i<j \leq N}(x_i-x_j) \equiv (x_1 \cdots x_n)^{N-n} \cdot \prod_{i<j \leq n}(x_i-x_j) \cdot (x_{n+1}^{N-n-1} x_{n+2}^{N-n-2}\cdots x_{N-1} + R)
\] and
\[
\schub_{w_0v}(x_1^{-1}, \ldots, x_N^{-1}) \equiv (x_1 \cdots x_n)^{-N+n}x_{n+1}^{-N+n+1} \cdots x_{N-1}^{-1} \cdot \schub_{\rho-\mu}(x_1^{-1}, \cdots, x_n^{-1})
\] modulo terms having degrees $>T$ and $>-T$ in variables $x_{n+1}, \ldots, x_N$ respectively. 
Thus $\schub_{w_0v}(x_1^{-1}, \ldots, x_N^{-1})\prod_{1 \leq i<j \leq N}(x_i-x_j)$ is equal to 
\[
\schub_{\rho-\mu}(x_1^{-1}, \cdots, x_n^{-1}) \cdot \prod_{1 \leq i<j \leq n}(x_i-x_j) \cdot (1+x_{n+1}^{-N+n+1} \cdots x_{N-1}^{-1}R)
\]
modulo terms with degree $>0$ in variables $x_{n+1}, \ldots, x_N$. 
Since $x_{n+1}, \ldots, x_N$ does not appear in $\schub_w$ and $x_{N-1}^{-N+n+1} \cdots x_{N-1}^{-1}R$ does not have a constant term, this shows 
\[\langle \schub_w, \schub_{w_0v}(x_1^{-1}, \ldots, x_N^{-1})\prod_{1 \leq i<j \leq N}(x_i-x_j) \rangle = \langle \schub_w, \schub_{\rho-\mu}(x_1^{-1}, \ldots, x_n^{-1})\prod_{1 \leq i<j \leq n}(x_i-x_j) \rangle. \] This, together with $(*)$, finishes the proof of Lemma \ref{dualbasislem}. 
\end{proof}

Let us come back to the proof of Lemma \ref{sansuu}. Essentially this is a ``Cauchy formula''
for the dual bases $\{\schub_\lmb\}$ and $\{\schub_{\rho-\mu}(x_1^{-1}, \ldots, x_n^{-1})\prod(x_i-x_j)\}$ appeared in Lemma \ref{dualbasislem}, 
but since we are dealing with an infinite-dimensional space a careful justification is needed. 
Let $c_{\alpha\beta}$ be the coefficient of $x^\alpha y^\beta$ in $\sum_{\kap \in \ZZ^n} \schub_{\rho-\kap}(x)\schub_\kap(y)$. 
We observe that if $c_{\rho-\mu, \nu} \neq 0$, then there exists some $\kap$ such that
$\rho-\mu \rtgeq \rho-\kap$ and $\nu \rtgeq \kap$, 
and so $\nu \rtgeq \kap \rtgeq \mu$. 
Thus $c_{\rho-\mu, \nu} = 0$ for $\nu \rtngeq \mu$. 
Using this as the base case, if we show $\sum_{g \in S_n} \sgn(g) c_{\rho-\mu, \nu-\rho+g\rho}=\delta_{\mu\nu}$, 
then we can show $c_{\rho-\mu, \nu}=\dim \ualg(\nplus)_{\nu-\mu}$ by induction on $\nu$
since $\sum_\kap \dim \ualg(\nplus)_\kap x^\kap = \prod_{i<j} (1-x_ix_j^{-1})^{-1}$ and $\prod_{i<j} (1-x_ix_j^{-1}) = \sum_{g \in S_n} \sgn(g) x^{\rho-g\rho}$. 
We show below the equivalent claim $\sum_{g \in S_n} \sgn(g) c_{\alpha, \beta+g\rho}=\delta_{\alpha,-\beta}$. 

Since $c_{\alpha,\beta+g\rho}=c_{\alpha+k\one,\beta+g\rho-k\one}$, we may assume that $-\beta \in \nonneg^n$. 
We may further assume, by replacing $\alpha$ and $\beta$ by $\alpha+k\one$ and $\beta-k\one$ for a sufficiently large $k$, that if $\kap \in \ZZ^n$ satisfies $\alpha \rtgeq \kap \rtgeq -\beta+\rho-g\rho$ for some $g \in S_n$ then $\kap \in \nonneg^n$ (this is possible by the remark at the end of Section \ref{schubmod}). 
Also it is sufficient to consider the case $|\alpha|=-|\beta|$. 
Let $d=|\alpha|$. Let $V$ be the space of all (ordinary) polynomials in $x_1, \ldots, x_n$ which are homogeneous of degree $d$. 
Equip $V$ with a bilinear form $\langle x^\sigma, x^\tau \rangle= \delta_{\sigma\tau}$. 
Then by Lemma \ref{dualbasislem} the bases $\{\schub_\kap : \kap \in \nonneg^n, |\kap|=d\}$ and 
$\{[\schub_{\rho-\kap}(x_1^{-1}, \ldots, x_n^{-1})\prod_{1 \leq i<j \leq n}(x_i-x_j)] : \kap \in \nonneg^n, |\kap|=d\}$ of $V$ are dual of each other; here for $f \in \ZZ[x_1^{\pm 1}, \ldots, x_n^{\pm 1}]$, $[f]$ is the sum of all terms in $f$ which do not contain any negative powers of $x_1, \ldots, x_n$. 
Thus we have
\begin{align*}
\sum_{\substack{\gamma \in \nonneg^n \\ |\gamma|=d}}x^\gamma y^\gamma
&\equiv
\sum_{\substack{\kap \in \nonneg^n \\ |\kap|=d}} \schub_\kap(x_1, \ldots, x_n) \schub_{\rho-\kap}(y_1^{-1}, \ldots, y_n^{-1})\prod_{1 \leq i<j \leq n}(y_i-y_j) \\
&= \left(\sum_{\substack{\kap \in \nonneg^n \\ |\kap|=d}} \schub_\kap(x_1, \ldots, x_n) \schub_{\rho-\kap}(y_1^{-1}, \ldots, y_n^{-1}) \right) \left( \sum_{g \in S_n} \sgn(g)y^{g\rho} \right) \quad \cdots (*)
\end{align*}
modulo terms containing some negative powers of some $y_i$ 
(note that for any finite-dimensional vector space $V$, 
the sum $\sum \phi_i \otimes \phi_i^* \in V \otimes V^*$ does not depend on the choice of dual bases $\{\phi_i\} \subset V, \{\phi_i^*\} \subset V^*$). 
Since $-\beta \in \nonneg^n$, the coefficient of $x^{\alpha}y^{-\beta}$ is equal for both side. The coefficient for the LHS is $\delta_{\alpha,-\beta}$. 
Moreover, if $\kap \in \ZZ^n$ and
$\schub_\kap(x_1, \ldots, x_n) \schub_{\rho-\kap}(y_1^{-1}, \ldots, y_n^{-1})$
contains some monomial of the form $x^{\alpha}y^{-\beta-g\rho}$ ($g \in S_n$) with nonzero coefficients, 
then such $\kap$ indeed satisfies $\kap \in \nonneg^n$ and thus appears in the first sum in $(*)$ above, 
since such $\kap$ must satisfy $\alpha \rtgeq \kap$ and $\beta+g\rho \rtgeq \rho-\kap$. 
So the coefficient of $x^{\alpha}y^{-\beta}$ in the RHS is the same as the coefficient of $x^{\alpha}y^{-\beta}$ in 
$\left( \sum_{\substack{\kap \in \ZZ^n \\ |\kap|=d}} \schub_\kap(x_1, \ldots, x_n) \schub_{\rho-\kap}(y_1^{-1}, \ldots, y_n^{-1})\right) \left( \sum_{g \in S_n} \sgn(g)y^{g\rho} \right)$. 
Since this coefficient is $\sum_{g \in S_n} \sgn(g) c_{\alpha, \beta+g\rho}$ we are done. 
\end{proof}
\begin{rem}
This proof, together with some results from the previous section, 
in fact shows that $\catc_{\leq \lmb}$ can be equipped with a structure of highest-weight category (\cite{CPS})
whose standards and costandards are $\smod_{\mu}$ ($\mu \leq \lmb$) and $\smod_{\rho-\nu}^* \otimes K_\rho$ ($\nu \leq \lmb$) respectively. 
The results in the next section is then a standard argument in the theory of highest-weight categories. 
I would like to thank Katsuyuki Naoi for giving the author this information. 

\end{rem}
From Lemma \ref{sugoilem}, we obtain the following corollary. 
This can be seen as an analog of ``Strong form of Polo's theorem'' (\cite[Theorem 3.2.2]{vdK}) for KP modules.
\begin{cor}
For $\lmb \in \ZZ^n$, $\mu, \nu \leq \lmb$ and $i \geq 1$, $\Ext^i_{\leq \lmb}(\smod_\mu, \smod_{\rho-\nu}^* \otimes K_\rho)=0$. 
\label{ext2}
\end{cor}
\begin{proof}
By Lemma \ref{sugoilem}, it suffices to prove $\Ext^i_{\leq \max\{\mu,\nu\}}(\smod_\mu, \smod_{\rho-\nu}^* \otimes K_\rho)=0$. 
If $\mu \geq \nu$, this follows from the projectivity of $\smod_\mu \in \catc_{\leq \mu}$
since $\smod_{\rho-\nu}^* \otimes K_\rho \in \catc_{\leq \mu}$. 
Otherwise it follows from the injectivity of $\smod_{\rho-\nu}^* \otimes K_\rho \in \catc_{\leq \nu}$ since $\smod_\mu \in \catc_{\leq \nu}$. 
\end{proof}

\section{Existence of KP filtrations}
\label{filtr}
Using the results obtained so far, 
we can obtain a criterion for a module to have a KP filtration, 
using the similar argument from \cite[\S 3]{vdK}. 
Hereafter, $\Ext^i(M, N)$ means $\Ext^i_{\leq \lmb}(M, N)$ for a suitable $\lmb$ (by Lemma \ref{sugoilem} this does not depend on the choice of $\lmb$). 
\begin{thm}
Let $\lmb \in \ZZ^n$, $M \in \catc_{\leq \lmb}$ and assume that $\Ext^1(M, \smod_{\rho-\mu}^* \otimes K_\rho)=0$ for all $\mu \leq \lmb$. 
Then $M$ has a filtration such that each of its subquotients is isomorphic to some $\smod_\nu$ $(\nu \leq \lmb)$.
\label{filtr_thm}
\end{thm}
Note that the converse obviously holds since $\Ext^1(\smod_\nu, \smod_{\rho-\mu}^* \otimes K_\rho)=0$. 
\begin{proof}
Let $\{\nu : \nu \leq \lmb\} = \{\nu^{(1)}<\nu^{(2)}<\cdots<\nu^{(r)}\}$. 
Let $M_i$ be the largest quotient of $M$ whose weights are in $\{\nu^{(1)},\ldots, \nu^{(i)}\}$ (so $M_0=0$ and $M_r = M$). 
By definition, we have a natural surjection $M_i \surj M_j$ for $i>j$. 
We show that $\Ker(M_i \surj M_{i-1})$ is a direct sum of some copies of $\smod_\nu$ by the induction on $i$. 
This will show that $M=M_r \surj M_{r-1} \surj \cdots \surj M_0=0$ gives a quotient filtration with desired property. 

Fix $i$ and let $\nu=\nu^{(i)}$. 
It is sufficient to show that $\Ker(M_i \surj M_{i-1})$ is the projective cover of its $\nu$-weight space $\Ker(M_i \surj M_{i-1})_\nu$ in $\catc_{\leq \nu}$, 
since $\smod_\nu$ is the projective cover of $K_\nu$ in $\catc_{\leq \nu}$.
Since $\Ker(M_i \surj M_{i-1})$ is generated by $\Ker(M_i \surj M_{i-1})_\nu$, 
it suffices to show that $\Ker(M_i \surj M_{i-1})$ is projective in $\catc_{\leq \nu}$, 
that is, $\Ext^1(\Ker(M_i \surj M_{i-1}),K_\mu)=0$ for all $\mu \leq \nu$. 

Let $\mu \leq \nu$. We have an exact sequence $\Ext^1(M, \smod_{\rho-\mu}^* \otimes K_\rho) \mor \Ext^1(\Ker(M \surj M_{i-1}), \smod_{\rho-\mu}^* \otimes K_\rho) \mor \Ext^2(M_{i-1},\smod_{\rho-\mu}^* \otimes K_\rho)$. 
Here $\Ext^1(M, \smod_{\rho-\mu}^* \otimes K_\rho)=0$ by the hypothesis. 
Moreover, $\Ext^2(M_{i-1}, \smod_{\rho-\mu}^* \otimes K_\rho)=0$ by Corollary \ref{ext2}, since $M_{i-1}$ has a filtration
by modules $\smod_\kappa$ ($\kappa < \nu$) by the induction hypothesis. 
Therefore $\Ext^1(\Ker(M \surj M_{i-1}), \smod_{\rho-\mu}^* \otimes K_\rho)=0$. 

We have an exact sequence $\Hom(\Ker(M \surj M_i), \smod_{\rho-\mu}^* \otimes K_\rho) \mor \Ext^1(\Ker(M_i \surj M_{i-1}), \smod_{\rho-\mu}^* \otimes K_\rho) \mor \Ext^1(\Ker(M \surj M_{i-1}), \smod_{\rho-\mu}^* \otimes K_\rho)=0$. 
But here $\Hom(\Ker(M \surj M_i), \smod_{\rho-\mu}^* \otimes K_\rho)=0$, since the weights of $\smod_{\rho-\mu}^* \otimes K_\rho$ are all less than or equal to $\mu$ and therefore $\leq \nu$, 
while $\Ker(M \surj M_i)$ is generated by the elements whose weights are $>\nu$.  
Therefore $\Ext^1(\Ker(M_i \surj M_{i-1}), \smod_{\rho-\mu}^* \otimes K_\rho)=0$. 

We have an exact sequence $\Hom(\Ker(M_i \surj M_{i-1}), (\smod_{\rho-\mu}^* \otimes K_\rho)/K_\mu) \mor \Ext^1(\Ker(M_i \surj M_{i-1}), K_\mu) \mor \Ext^1(\Ker(M_i \surj M_{i-1}), \smod_{\rho-\mu}^* \otimes K_\rho)=0$. 
But since the weights of $(\smod_{\rho-\mu}^* \otimes K_\rho)/K_\mu$ are strictly less than $\mu$ and thus $< \nu$ while $\Ker(M_i \surj M_{i-1})$ is generated by its $\nu$-weight space, $\Hom(\Ker(M_i \surj M_{i-1}), (\smod_{\rho-\mu}^* \otimes K_\rho)/K_\mu)=0$.  
So $\Ext^1(\Ker(M_i \surj M_{i-1}), K_\mu)=0$ and we are done. 
\end{proof}

Another criterion for filtration can be also derived: 
\begin{thm}
Let $\lmb \in \ZZ^n$ and $M \in \catc_{\leq \lmb}$. 
Then $\ch(M) \leq \sum_{\nu \leq \lmb} \dim_K(\Hom(M,\smod_{\rho-\nu}^* \otimes K_\rho)) \schub_\nu$, 
and the equality holds if and only if $M$ has a filtration such that each of its subquotients is isomorphic to some $\smod_\nu$ $(\nu \leq \lmb)$. Here $\sum a_\alpha x^\alpha \leq \sum b_\alpha x^\alpha$ is defined as $a_\alpha \leq b_\alpha$ $(\forall \alpha)$. 
\label{filtr_char}
\end{thm}
\begin{proof}
Let $\{\nu : \nu \leq \lmb\} = \{\nu^{(1)}<\nu^{(2)}<\cdots<\nu^{(r)}\}$. 
By the proof of Theorem \ref{filtr_thm}, $M$ has a desired filtration if and only if $\Ker(M_i \surj M_{i-1})$ is a direct sum of some copies of $\smod_{\nu^{(i)}}$, 
where $M_i$ is the largest quotient of $M$ whose weight are in 
$\{ \nu^{(1)}, \ldots, \nu^{(i)}\}$. 

We have $\ch(M) = \sum_i \ch(\Ker(M_i \surj M_{i-1}))$. 
Since $\Ker(M_i \surj M_{i-1})$ is generated by its $\nu^{(i)}$-weight space $(M_i)_{\nu^{(i)}}$, 
if we let $d_i$ denote the dimension of this weight space, 
we have a surjection from $(P_{\nu^{(i)}}^{\leq \nu^{(i)}})^{\oplus d_i}$ to $\Ker(M_i \surj M_{i-1})$. 
We have seen in Proposition \ref{projinj} that $P_{\nu^{(i)}}^{\leq \nu^{(i)}} \cong \smod_{\nu^{(i)}}$. 
Thus $\ch(M) \leq \sum_i \dim((M_i)_{\nu^{(i)}})\schub_{\nu^{(i)}}$ and the equality holds when and only when each kernel is a direct sum of some copies of $\smod_{\nu^{(i)}}$, i.e. $M$ has a desired filtration. 

For each $i$, we have 
$\Hom(M_i, \smod_{\rho-\nu}^* \otimes K_\rho) \cong \Hom(\smod_{\rho-\nu}, M_i^* \otimes K_\rho)
\cong (M_i^* \otimes K_\rho)_{\rho-\nu}
\cong ((M_i)_{\nu})^*
$ where $\nu=\nu^{(i)}$, 
since $M_i^* \otimes K_\rho \in \catc_{\{\rho-\mu : \mu \leq \nu\}}=\catc_{\leq' \rho-\nu}$
and $\smod_{\rho-\nu}$ is the projective cover of $K_{\rho-\nu}$ in this category. 
Thus the theorem follows. 
\end{proof}

\section{Questions}
\label{questions}
\begin{Q}
For $\lmb, \mu \in \ZZ^n$, does $\smod_\lmb \otimes \smod_\mu$ have a KP filtration?
\label{tensor_q}
\end{Q}
By the criteria obtained above, this question is equivalent to ask:
\begin{itemize}
\item whether $\Ext^1(\smod_\lmb \otimes \smod_\mu, \smod_\nu^* \otimes K_\rho)=0$ or not, or
\item whether the dimension of $\Hom(\smod_\lmb \otimes \smod_\mu, \smod_{\rho-\nu}^* \otimes K_\rho)$
is equal to the coefficient of $\schub_\nu$
in the expansion of $\schub_\lmb\schub_\mu$ into a linear combination of Schubert polynomials. 
\end{itemize}

\begin{Q}
Let $s_\sigma$ denote the Schur functor corresponding to a partition $\sigma$ and let $\lmb \in \ZZ^n$. 
Then, does $s_\sigma(\smod_\lmb)$ have a KP filtration?
\label{pleth_q}
\end{Q}
As explained in the introduction, positive answer for this question implies that
the ``plethysm'' $s_\sigma[\schub_\lmb]$ is a positive sum of Schubert polynomials. 

We note the following connection between these two problems. 
\begin{prop}
Suppose that the answer to Question \ref{tensor_q} is yes. 
Then the answer to Question \ref{pleth_q} is yes. 
\end{prop}
\begin{proof}
By iteratively using \ref{tensor_q}, we see that $\smod_{\lmb^{(1)}} \otimes \cdots \otimes \smod_{\lmb^{(r)}}$ has a KP filtration
for any $\lmb^{(1)}, \ldots, \lmb^{(r)} \in \ZZ^n$. 
Especially, $(\smod_\lmb)^{\otimes k}$ has a KP filtration for any $\lmb$ and $k$. 
Therefore $\Ext^1((\smod_\lmb)^{\otimes k}, \smod_\nu^* \otimes K_\rho)=0$ for any $\nu$. 
Since $s_\sigma(\smod_\lmb)$ is a direct sum factor of $(\smod_\lmb)^{\otimes |\sigma|}$, 
$\Ext^1(s_\sigma(\smod_\lmb), \smod_\nu^* \otimes K_\rho)=0$. 
Thus $s_\sigma(\smod_\lmb)$ has a KP filtration by Theorem \ref{filtr_thm}. 
\end{proof}

\vspace{2ex}
\noindent\textbf{Note.} in a subsequent work the author gave positive answers to both of the questions above: see \cite{W2}. 
\nocite{*}

\bibliographystyle{plain}

\end{document}